\documentclass[reqno,draft]{amsart}

\date{\today}
\theoremstyle{plain}
\newtheorem{thm}{Theorem}[section]

\newtheorem{lem}[thm]{Lemma}

\theoremstyle{definition}

\theoremstyle{remark}
\newtheorem{rem}{Remark}[section]
\numberwithin{equation}{section}
\renewcommand{\theequation}{\thesection.\arabic{equation}}

\begin{document}
\allowdisplaybreaks

\title[quasineutral  limit of the electro-diffusion model]
{Quasineutral  limit of the electro-diffusion model arising in
Electrohydrodynamics}

 \author{Fucai Li}
\address{Department  of Mathematics,  Nanjing University, Nanjing 210093, P.R. China}
 \email{fli@nju.edu.cn}





\begin{abstract}
 The electro-diffusion model, which arises in
electrohydrodynamics, is a coupling between the
Nernst-Planck-Poisson system and the incompressible Navier-Stokes
equations.  For the generally smooth doping profile, the
quasineutral limit (zero-Debye-length limit) is justified rigorously
in Sobolev norm uniformly in time.  The proof is based on the
elaborate energy analysis and the key
point is to establish the uniform estimates with respect to the
scaled Debye length.
\end{abstract}
\maketitle

\renewcommand{\theequation}{\thesection.\arabic{equation}}
\setcounter{equation}{0}
\section{Introduction and Main Results} \label{S1}
In this paper we consider a   model describing ionic concentrations,
electric potential, and velocity field in an electrolytic solution.
This model is a coupling between the Nernst-Planck-Poisson system
and the Navier-Stokes equations \cite{R90,Ro,CF,J02}. The (rescaled)
system takes the form
\begin{align}
 & n^\lambda_t=\text{div}(\nabla n^\lambda -n^\lambda
  \nabla\Phi^\lambda -n^\lambda v^\lambda), \label{ea1}\\
  & p^\lambda_t=\text{div}(\nabla p^\lambda +p^\lambda
  \nabla\Phi^\lambda -p^\lambda v^\lambda),\label{ea2}\\
&\lambda^2\Delta \Phi^\lambda =n^\lambda -p^\lambda
-D(x),\label{ea3}\\
&v^\lambda_t+v^\lambda\cdot \nabla v^\lambda +\nabla \pi^\lambda-\mu
\Delta v^\lambda
=(n^\lambda -p^\lambda)\nabla\Phi^\lambda, \label{ea4}\\
& \text{div}v^\lambda =0   \label{ea5}
\end{align}
with initial data
\begin{equation}\label{ea6}
   n^\lambda(x,0)=n^\lambda_0(x), \quad
   p^\lambda(x,0)=p^\lambda_0(x), \quad
   v^\lambda(x,0)=v^\lambda_0(x), \ \ \ x\in \mathbb{T}^3,
\end{equation}
where $\mathbb{T}^3 $ is the periodic  domain in $\mathbb{R}^3$,
 ${n}^\lambda$ and $p^\lambda$ denote the negative and positive
charges respectively, $\Phi^\lambda$  the electric field,
$v^\lambda$   the velocity of the electrolyte, and $\pi^\lambda$
  the fluid pressure. The parameter $\lambda>0$ denotes the
scaled Debye length and $\mu>0$  the dynamic viscosity.
$D(x)$ is a given function and models the doping profile.

 Usually in electrolytes the
Debye length is much smaller compared the others quantities, and the
electrolytes is almost electrically neutral. Under the assumption of
space charge neutrality, i.e. $\lambda =0$,  we formally arrive at
the following quasineutral Nernst-Planck-Navier-Stokes system
\begin{align}
 & n_t=\text{div}(\nabla n+n
  \mathcal{E} -nv), \label{ea7}\\
  & p_t=\text{div}(\nabla p -p\,
  \mathcal{E} -pv), \label{ea8}\\
&n-p-D(x)=0,\label{ea9}\\
&v_t+v\cdot \nabla v+\nabla \pi -\mu\Delta v
=-(n  -p )\mathcal{E},  \label{ea10}\\
& \text{div}v =0,  \label{ea11}
\end{align}
where we assume that the  limits $n^\lambda \rightarrow n$, $
p^\lambda \rightarrow p$, $ v^\lambda \rightarrow v$,
$-\nabla\Phi^\lambda \equiv E^\lambda \rightarrow \mathcal{E}$ exist
as $\lambda\rightarrow 0^+$.

The purpose of this paper is to justify the above  limit rigorously
for sufficiently smooth solutions to the system
\eqref{ea1}-\eqref{ea5}.

Since the incompressible Navier-Stokes equations
\eqref{ea4}-\eqref{ea5} are  involved in the system
\eqref{ea1}-\eqref{ea5}, it is well known that whether the global
classical solution for general initial data exists or not is open
for three spatial dimensional case and only local classic solution
is available. For example, in \cite{J02}, Jerome studied the Cauchy
problem of  the system \eqref{ea1}-\eqref{ea5} and established the
local existence of   unique smooth solution for smooth initial data.
The local existence of   unique smooth solution to the
incompressible Navier-Stokes equations can be obtained by standard
method, see \cite{L69,T01}.

  The local existence of   unique smooth solution to the limiting system
\eqref{ea7}-\eqref{ea11}
with initial smooth data
\begin{equation} \label{ea122}
n(x,t=0)=n_0(x),\ \  p(x,t=0)= p_0(x), \ \  v(x,t=0)= v_0(x)
\end{equation}
can be obtained by the similar arguments to
those stated in \cite{J02}. Since we are interested in the
quasineutral limit of the system \eqref{ea1}-\eqref{ea5}, we omit
the detail here.

In this paper we assume that the doping profile is a smooth
(sign-changing) function and the initial data $n^\lambda_0(x),
 p^\lambda_0(x)$ and $ v^\lambda_0(x)$ are smooth functions satisfying
 \begin{equation}\label{ea123}
   \int (n^\lambda_0(x)-
 p^\lambda_0(x)-D(x)) dx =0, \quad  \int  v^\lambda_0(x)dx=0.
 \end{equation}

The  main result of this paper can be stated as follows:

\begin{thm}\label{th}
Let $(n^\lambda, p^\lambda, E^\lambda, v^\lambda)$,
$E^\lambda=-\nabla \Phi^\lambda$ be the unique local smooth solution
to the system \eqref{ea1}-\eqref{ea5} with initial data \eqref{ea6}
on $\mathbb{T}^3\times [0,T_*)$ for some $0<T_*\leq\infty$. Let
$(n,p,\mathcal{E},v), \mathcal{E}=-\nabla \Phi$ be the unique smooth
solution to the limiting system  \eqref{ea7}-\eqref{ea11} with
initial data \eqref{ea122} on $\mathbb{T}^3\times [0,T_0)$ for some
$0<  T_0\leq  +\infty$  satisfying  $n+p\geq \kappa_0>0$, where
$\kappa_0$ is a positive constant. Suppose  that initial data
satisfy \eqref{ea123} and
\begin{equation}\label{ea26}
  n^\lambda_0(x)=n_0(x), \quad p^\lambda_0(x)=p_0(x)+\lambda^2{\rm div}\mathcal{E}(t=0),\quad
  v^\lambda_0(x)=v_0(x).
\end{equation}
Then, for any $T\in
(0,\min\{T_0,T_*\})$, there exist positive constants $K$ and
$\lambda_0, \lambda_0\ll 1$, such that, for any $\lambda\in
(0,\lambda_0)$,
\begin{align}\label{ea27}
 &\sup\limits_{0\leq t\leq T}\Big\{||(\tilde{n}^\lambda,\tilde{p}^\lambda, \tilde{E}^\lambda, \tilde{v}^\lambda)(t)||_{H^1}
 +||(\tilde{n}^\lambda_t,\tilde{p}^\lambda_t,\tilde{v}^\lambda_t)(t)||_{L^2}
 \nonumber\\
 & \qquad\qquad +\lambda ||\tilde{E}^\lambda(t)||_{H^2}+\lambda ||\tilde{E}^\lambda_t(t)||_{H^1}\Big\}\leq K\lambda^{1-\sigma/2}
\end{align}
for any $\sigma\in (0,2)$, independent of $\lambda$. Here
$\tilde{n}^\lambda=n^\lambda-n,
 \tilde{p}^\lambda=p^\lambda-p, \tilde{E}^\lambda=E^\lambda-\mathcal{E}, $ and $\tilde{v}^\lambda=v^\lambda-v$.
\end{thm}

\begin{rem}
 In this paper we deal with the three spatial dimensional case, if the
 problem \eqref{ea1}-\eqref{ea5} is considered in  two dimensional space,
both the problem \eqref{ea1}-\eqref{ea5} and the limiting problem
\eqref{ea7}-\eqref{ea11} enjoy
  global smooth solutions, thus we can   obtain
 a similar result to that stated in Theorem \ref{th} (in fact much easier).
\end{rem}

\begin{rem}
 If the assumption \eqref{ea26} does not hold, we need to consider the   initial layers.
 On the other hand, if we consider the system \eqref{ea1}-\eqref{ea5} on the smooth bounded domain in
 $\mathbb{R}^3$, the boundary layers may appear. These issues will be   studied in  the future.
\end{rem}

The main difficulty in dealing with the quasineutral limits is the
oscillatory behavior of the electric field (the Poisson equation
becomes an algebraic equation in the limit). Usually it is difficult
to obtain uniform estimates on the electric field with respect to
the Debye length $\lambda$ due to a possible vacuum set of density.
To overcome this difficulty, we introduce the following
$\lambda$-weighted Lyapunov-type functionals
\begin{align}\label{es15}
\Gamma^\lambda (t)\equiv& ||(\tilde{z}^\lambda, \nabla
\tilde{z}^\lambda,\Delta \tilde{z}^\lambda,
\tilde{z}^\lambda_t,\nabla \tilde{z}^\lambda_t)||^2
+||(\tilde{v}^\lambda, \nabla \tilde{v}^\lambda,\Delta \tilde{v}^\lambda, \tilde{v}^\lambda_t,\nabla \tilde{v}^\lambda_t)||^2\nonumber\\
&+\lambda^2||(\tilde{E}^\lambda, {\rm div}  \tilde{E}^\lambda,
\nabla {\rm div} \tilde{E}^\lambda, \tilde{E}^\lambda_t, {\rm
div}  \tilde{E}^\lambda_t)||^2+ ||(\tilde{E}^\lambda, {\rm
div}\tilde{E}^\lambda)||^2
\end{align}
and
\begin{align}\label{es16}
G^\lambda(t)\equiv||(\Delta \tilde{z}_t, \Delta\tilde{v}_t,
\tilde{E}^\lambda_t, {\rm
div}\tilde{E}^\lambda_t)||^2_{L^2}+\lambda^2||\nabla {\rm
div}\tilde{E}^\lambda_t||^2,
\end{align}
where $\tilde{z}^\lambda=\tilde{n}^\lambda +\tilde{p}^\lambda$,
 $\tilde{n}^\lambda=n^\lambda-n,
\tilde{p}^\lambda=p^\lambda-p,  \tilde{v}^\lambda=v^\lambda-v,
 \tilde{E}^\lambda=E^\lambda-\mathcal{E}$ and
$(\tilde{n}^\lambda,\tilde{p}^\lambda,\tilde{E}^\lambda,\tilde{v}^\lambda)$
denotes the difference between the solution to the system
\eqref{ea1}-\eqref{ea5} and the solution to the limiting system
\eqref{ea7}-\eqref{ea11}, see Section \ref{S2} below for details. By
a careful energy method, we can prove the following entropy
production integration inequality
\begin{align}\label{ns16}
\Gamma^\lambda (t)+\int^t_0 G^\lambda(s)ds\leq &
K\,\Gamma^\lambda(t=0)+K\lambda^q+K(\Gamma^\lambda(t))^r+K\int^t_0
\Gamma^\lambda(s)G^\lambda(s)ds\nonumber\\
& +K\int^t_0
\big[ \Gamma^\lambda(s) +(\Gamma^\lambda(s))^l\big]ds,
\quad t\geq 0
\end{align}
for  some positive constants $q,r,K$ and $l$, independent of $\lambda$, which
implies our desired convergence result by the assumption of small initial data
$\Gamma^\lambda (0)$.

\begin{rem}
 The inequality \eqref{ns16} is a generalized Gronwall's type with an extra integration term where
the integrand function is the production of the entropy and the entropy-dissipation. Hence  \eqref{ns16} is
called as the entropy production integration inequality.
\end{rem}

\begin{rem}
The $\lambda$-weighted Lyapunov-type functional \eqref{es15} and
\eqref{es16} is motivated by  \cite{HW06,W06}, where the
quasineutral limit of drift-diffusion-Poisson model for
semiconductor was studied. However, in our case the incompressible
Navier-Stokes equations are involved and the more refined energy
analysis  is needed. We believe that those  $\lambda$-weighted
Lyapunov-type   energy functionals  can also be used to deal with
the quasineutral limit problem of  other  mathematical models
involving in Navier-Stokes equations,   for example,   the
mathematical model for the deformation of electrolyte droplets:
\begin{align*}
\rho(u_t +u\cdot \nabla u)+\pi& =\nu \Delta u+(n-p)\nabla V-\nabla \cdot (\nabla \phi \otimes \nabla \phi),\\
\nabla \cdot u &=0,\\
n_t+u\cdot \nabla n& =\nabla \cdot (D_n\nabla n-\mu_nn\nabla V+M n\nabla \phi),\\
p_t+u\cdot \nabla p& =\nabla \cdot (D_p\nabla p-\mu_pp\nabla V+M p\nabla \phi),\\
\nabla \cdot (\lambda  \nabla V)& =n-p,\\
\phi_t+u\cdot \nabla\phi & =\gamma (\Delta\phi-\eta^{-2}W'(\phi)),
\end{align*}
where $\gamma, \nu, \eta, D_n,D_p,\mu_n,\mu_p $ and $M $ are positive constants,
see \cite{RLZ} for the detailed description on this model.
\end{rem}

We point out  that the quasineutral limit is a well-known
challenging and physically complex modeling problem for fluid
dynamic models and for kinetic models of semiconductors and plasmas
and other fields. In both cases, there only exist partial results.
For time-dependent transport models, the limit $\lambda \rightarrow
0$ has be performed for the Vlasov-Poisson system by
Brenier~\cite{B} and Masmoudi~\cite{M},   and for the
Vlasov-Poisson-Fokker-Planck system by Hsiao et al~\cite{HLW2,HLW3},
respectively. For the fluid dynamic model, the
drift-diffusion-Poisson system is investigate by Gasser et
al~\cite{GHMW,GLMS} and J\"{u}ngel and Peng \cite{JP}, and for the
Euler-Poisson system by Cordier and Grenier ~\cite{CG} and
Wang~\cite{W04}. Recently, Wang et al  \cite{W06,HW06,WXM} extends
some results cited above for the general doping profiles, the main
idea is to control  the strong nonlinear oscillations caused by
small Debye length by the interaction of the physically motivated
entropy and the entropy dissipation. For the Navier-Stokes-Poisson
system, Wang \cite{W04,WaJ06} obtained the convergence of the
Navier-Stokes-Poisson system to the incompressible Euler equations.
Ju et al \cite{JLW} obtained the convergence of weak solutions of
the Navier-Stokes-Poisson system to the
  strong   solutions of incompressible Navier-Stokes
equations. Donatelli and  Marcati \cite{DM} studied  the
quasineutral-type   limit for the Navier-Stokes-Poisson system with large initial
data in the whole space $\mathbb{R}^3$ through
 the coupling of the zero-Debye-length limit and the low Mach number limit.

We mention  that there are a few other mathematical results on the
system \eqref{ea1}-\eqref{ea5}.  Jerome~\cite{J02} obtained the
inviscid limit ($\mu \rightarrow 0)$ of the system
\eqref{ea1}-\eqref{ea5}. Cimatti and Fragal\`{a}~\cite{CF} obtained
the unique weak solution to the system \eqref{ea1}-\eqref{ea5} with
Neumann boundary condition and the  asymptotic behavior of solution
when it is a small perturbation of the trivial solution for the
stationary problem. Feireisl~\cite{F95} studied the system
\eqref{ea1}-\eqref{ea5} in periodic case without the diffusion terms
in the first two equations and obtained the existence of weak
solution.

Before ending this introduction, we give some notations. We denote
$||\cdot||$  the standard $L^2$ norm with respect to $x$, $H^k$
the standard Sobolev space $W^{k,2}$, and $||\cdot||_{H^k}$ the
corresponding  norm.  The notation $||(A_1,A_2, \cdots,A_n)||^2$
means the summation of $||A_i||^2,i=1,\cdots,n$,
 and it   also applies to    other norms.
We use $c_i$, $\delta_i$, $\epsilon$,
  $K_\epsilon$,  $K_i$, and $K$ to denote the constants which are
independent of $\lambda$ and may be changed from line to line. We
also omit in integral spatial domain $\mathbb{T}^3$ for convenience.
In Section \ref{S2}, we give some basic energy estimates of the error
system, and the proof of Theorem \ref{th} is given in Section \ref{S3}.

\renewcommand{\theequation}{\thesection.\arabic{equation}}
\setcounter{equation}{0}
\section{The energy estimates} \label{S2}

In this section we obtain some energy estimates needed to prove our
result. To this end, we first derive the error system from the original system \eqref{ea1}-\eqref{ea5}
and the limiting system \eqref{ea7}-\eqref{ea11} as follows.
Setting $\tilde{n}^\lambda=n^\lambda-n,   \tilde{p}^\lambda=p^\lambda-p,
\tilde{v}^\lambda=v^\lambda-v,  \tilde{\pi}^\lambda=\pi^\lambda-\pi,
 \tilde{E}^\lambda=E^\lambda-\mathcal{E}$ with $ \tilde{E}^\lambda=-\nabla \tilde{\Phi}^\lambda,
 {E}^\lambda=-\nabla {\Phi}^\lambda, \mathcal{E}=-\nabla\Phi $ and
 $\tilde{\Phi}^\lambda={\Phi}^\lambda-{\Phi}$,
   using  the system
\eqref{ea1}-\eqref{ea5} and the system \eqref{ea7}-\eqref{ea11}, we
  obtain
\begin{align}
 & \tilde{n}^\lambda_t=\text{div}(\nabla \tilde{n}^\lambda +n\tilde{E}^\lambda +\tilde{n}^\lambda
  (\tilde{E}^\lambda+\mathcal{E})-\tilde{n}^\lambda(\tilde{v}^\lambda+v)-n\tilde{v}^\lambda), \label{ea12}\\
  & \tilde{p}^\lambda_t=\text{div}(\nabla \tilde{p}^\lambda -p\tilde{E}^\lambda -\tilde{p}^\lambda
  (\tilde{E}^\lambda+\mathcal{E})-\tilde{p}^\lambda(\tilde{v}^\lambda+v)-p\tilde{v}^\lambda) ,\label{ea13}\\
&-\lambda^2\text{div}\tilde{E}^\lambda =\tilde{n}^\lambda
-\tilde{p}^\lambda
+\lambda^2\text{div}\mathcal{E},\label{ea14}\\
&\tilde{v}^\lambda_t+\tilde{v}^\lambda\cdot \nabla
\tilde{v}^\lambda+{v}\cdot \nabla
\tilde{v}^\lambda+\tilde{v}^\lambda\cdot \nabla  {v} +\nabla
\tilde{\pi}^\lambda-\mu \Delta \tilde{v}^\lambda\nonumber\\
&\qquad =-(\tilde{n}^\lambda -\tilde{p}^\lambda)(\tilde{E}^\lambda +\mathcal{E})-(n-p)\tilde{E}^\lambda, \label{ea15}\\
& \text{div}\tilde{v}^\lambda =0.  \label{ea16}
\end{align}

Set ${Z}=n+p$, then \eqref{ea7}-\eqref{ea11} is reduced to
\begin{align*}
 Z_t&=\text{div}(\nabla Z+D\mathcal{E}-Zv),\\
 0&=\text{div}(\nabla D +Z\mathcal{E}-Dv),\\
 v_t+v\cdot \nabla v&=-\nabla\pi +\mu \Delta v-D\mathcal{E},\\
 \text{div}v&=0
\end{align*}
with initial data $ Z(x,0)=n_0(x)+p_0(x)$ and $ v(x,0)=v_0(x).$

To obtain the desired energy estimates, we introduce new error
variable $\tilde{z}^\lambda=\tilde{n}^\lambda +\tilde{p}^\lambda$,
by the Poisson equation~\eqref{ea14},  we have
\begin{align}\label{ea21}
  \tilde{n}^\lambda =\frac{\tilde{z}^\lambda
  -\lambda^2\text{div}\tilde{E}^\lambda
  -\lambda^2\text{div}\mathcal{E}}{2},\qquad
    \tilde{p}^\lambda =\frac{\tilde{z}^\lambda
  +\lambda^2\text{div}\tilde{E}^\lambda
  +\lambda^2\text{div}\mathcal{E}}{2}.
\end{align}
Thus  the error system can be reduced to the following equivalent system
\begin{align}
&\tilde{z}_t^\lambda=\text{div}(\nabla \tilde{z}^\lambda
+D\tilde{E}^\lambda)-\lambda^2\text{div}(\mathcal{E}\text{div}\tilde{E}^\lambda+\tilde{E}^\lambda
\text{div}\mathcal{E})-\lambda^2\text{div}(\mathcal{E}\text{div}\mathcal{E})\nonumber\\
&\qquad \quad-\text{div}(\tilde{z}^\lambda\tilde{v}^\lambda
+\tilde{z}^\lambda v)
-\text{div}(Z\tilde{v}^\lambda)-\lambda^2\text{div}(\tilde{E}^\lambda\text{div}\tilde{E}^\lambda),\label{ea22}\\
& \lambda^2[\partial_t\text{div}\tilde{E}^\lambda-\text{div}(\nabla
\text{div}\tilde{E}^\lambda)]
+\text{div}(Z\tilde{E}^\lambda)\nonumber\\
& \qquad =-\lambda^2(\partial_t\text{div}\mathcal{E}-\Delta
\text{div}\mathcal{E})-\text{div}(\tilde{z}^\lambda\mathcal{E})
-\text{div}(\tilde{z}^\lambda\tilde{E}^\lambda)+\text{div}(D\tilde{v}^\lambda)\nonumber\\
&\qquad\quad
-\lambda^2\text{div}(\tilde{v}^\lambda\text{div}\mathcal{E}+v\text{div}\mathcal{E})
-\lambda^2\text{div}(\tilde{v}^\lambda{\rm div}\tilde{E}^\lambda+v\text{div}\tilde{E}^\lambda),\label{ea23}\\
&\tilde{v}^\lambda_t+\tilde{v}^\lambda\cdot \nabla
\tilde{v}^\lambda+{v}\cdot \nabla
\tilde{v}^\lambda+\tilde{v}^\lambda\cdot \nabla  {v} +\nabla
\tilde{\pi}^\lambda-\mu \Delta \tilde{v}^\lambda\nonumber\\
& \qquad =\lambda^2\tilde{E}^\lambda\text{div}\mathcal{E}
+\lambda^2\mathcal{E}\text{div}\tilde{E}^\lambda+\lambda^2\mathcal{E}\text{div}\mathcal{E}-D\tilde{E}^\lambda
+\lambda^2\tilde{E}^\lambda\text{div}\tilde{E}^\lambda,\label{ea24}\\
& \text{div}\tilde{v}^\lambda=0.\label{ea25}
\end{align}

For  the sake of notional simplicity, we set
$\tilde{\mathbf{w}}^\lambda=(\tilde{z}^\lambda, \tilde{E}^\lambda,
\tilde{v}^\lambda)$ and define the following  $\lambda$-weighted 
Sobolev's norm
\begin{equation}\label{eb1}
|||\tilde{\mathbf{w}}^\lambda|||^2\equiv||(\tilde{z}^\lambda,
\lambda\tilde{E}^\lambda,
\tilde{v})||^2_{H^2}+||( \tilde{z}^\lambda_t,
\lambda  \tilde{E}^\lambda_t,
 \tilde{v}^\lambda_t)||^2_{H^1}+||\tilde{E}^\lambda||^2_{H^1}.
\end{equation}

The following basic inequality can be derived from Sobolev's
embedding theorem and  will be used frequently in this paper.
\begin{lem}\label{L1}
 For $f,g\in H^1(\mathbb{T}^3)$, we have
\begin{equation}\label{eb2}
||fg||_{L^2}\leq ||f||_{L^4}\cdot||g||_{L^4}\leq K||f||_{H^1}\cdot
||g||_{H^1}.
\end{equation}
\end{lem}

\subsection{Low order estimates}
In this subsection, we derive  the low order energy estimates from
the error system \eqref{ea22}-\eqref{ea25}. The first estimate is
the $L^\infty_t(L^2_x)$ norm of $(\tilde{z}^\lambda,
\tilde{v}^\lambda, \tilde{E}^\lambda)$.
\begin{lem}\label{L2}
   Under the assumptions of Theorem \ref{th}, we have
   \begin{align}\label{eb3}
&||\tilde{z}^\lambda||^2 +||\tilde{v}^\lambda||^2 +\lambda^2||\tilde{E}^\lambda||^2 \nonumber\\
&\quad + \int^t_0 \big(||\nabla \tilde{z}^\lambda||^2 +\lambda^2
||{\rm div}\tilde{E}^\lambda||^2 +||\tilde{E}^\lambda||^2
+||\nabla\tilde{v}^\lambda||^2
 \big)(s)ds\nonumber\\
& \leq
 K(||\tilde{z}^\lambda||^2+||\tilde{v}^\lambda||^2 +\lambda^2||\tilde{E}^\lambda||^2 )(t=0)
   \nonumber\\
 & \quad+K\int^t_0\big(||\tilde{z}^\lambda||^2 +||\tilde{v}^\lambda||^2  +|||\tilde{\mathbf{w}}^\lambda|||^4\big)(s)ds+K\lambda^4.
  \end{align}
\end{lem}

\begin{proof}
  Multiplying \eqref{ea22} by $\tilde{z}^\lambda$ and integrating
  the resulting equation over $\mathbb{T}^3$ with respect to $x$, we get
\begin{align}\label{eb4}
 & \frac{1}{2}\frac{d}{dt}||\tilde{z}^\lambda||^2+||\nabla
  \tilde{z}^\lambda||^2\nonumber\\
  &= -\int D\tilde{E}^\lambda\nabla\tilde{z}^\lambda dx
  +\lambda^2\int \mathcal{E}{\rm div}\mathcal{E}\nabla \tilde{z}^\lambda dx+\int v \tilde{z}^\lambda \nabla \tilde{z}^\lambda dx\nonumber\\
  &\quad +\lambda^2\int (\tilde{E}^\lambda{\rm
  div}\mathcal{E}+\mathcal{E}{\rm
  div}\tilde{E}^\lambda)\nabla \tilde{z}^\lambda dx+\int
  Z\tilde{v}^\lambda \nabla \tilde{z}^\lambda dx \nonumber\\
 & \quad +\int \tilde{z}^\lambda \tilde{v}^\lambda \nabla \tilde{z}^\lambda dx+\lambda^2\int {\rm div}\tilde{E}^\lambda \tilde{E}^\lambda \nabla
  \tilde{z}^\lambda dx.
\end{align}
We estimate the terms on the right-hand side of \eqref{eb4}. For the first
five terms, by Cauchy-Schwartz's inequality and using the regularity
of $D, \mathcal{E}, v  $ and $Z$, which can be bounded by
\begin{equation}\label{eb5}
  \epsilon ||\nabla \tilde{z}^\lambda||^2+K_\epsilon
  ||(\tilde{E}^\lambda,\tilde{v}^\lambda
  ,\tilde{z}^\lambda)||^2+
  K_\epsilon \lambda^4||(\tilde{E}^\lambda, {\rm
  div}\tilde{E}^\lambda)||^2+K_\epsilon \lambda^4.
\end{equation}
For the sixth nonlinear term, by  Cauchy-Schwartz's inequality and
Sobolev's embedding $H^2(\mathbb{T}^3)\hookrightarrow
L^\infty(\mathbb{T}^3)$, we get
\begin{align}\label{eb6}
\int \tilde{z}^\lambda \tilde{v}^\lambda \nabla \tilde{z}^\lambda dx
&\leq \epsilon||\nabla \tilde{z}^\lambda||^2+K_\epsilon
||\tilde{v}^\lambda \tilde{z}^\lambda||^2\nonumber\\
&\leq \epsilon||\nabla \tilde{z}^\lambda||^2+K_\epsilon
||\tilde{v}^\lambda||_{L^\infty}^2|| \tilde{z}^\lambda||^2\nonumber\\
&\leq \epsilon||\nabla \tilde{z}^\lambda||^2+K_\epsilon
||\tilde{v}^\lambda||_{H^2}^2|| \tilde{z}^\lambda||^2\nonumber\\
&\leq \epsilon||\nabla \tilde{z}^\lambda||^2+K_\epsilon
|||\tilde{\mathbf{w}}^\lambda|||^4.
\end{align}
Similarly, for the last nonlinear term, we have
\begin{align}\label{eb7}
\lambda^2\int {\rm div}\tilde{E}^\lambda \tilde{E}^\lambda \nabla
  \tilde{z}^\lambda dx
&\leq \epsilon||\nabla \tilde{z}^\lambda||^2+K_\epsilon
\lambda^4||\tilde{E}^\lambda {\rm div}\tilde{E}^\lambda||^2\nonumber\\
&\leq \epsilon||\nabla \tilde{z}^\lambda||^2+K_\epsilon
\lambda^4||\tilde{E}^\lambda||_{L^\infty}^2 ||{\rm div}\tilde{E}^\lambda||^2\nonumber\\
&\leq \epsilon||\nabla \tilde{z}^\lambda||^2+K_\epsilon
\lambda^4||\tilde{E}^\lambda||_{H^2}^2 ||{\rm div}\tilde{E}^\lambda||^2\nonumber\\
&\leq \epsilon||\nabla \tilde{z}^\lambda||^2+K_\epsilon
|||\tilde{\mathbf{w}}^\lambda|||^4.
\end{align}
Thus, putting \eqref{eb4}-\eqref{eb7} together and taking $\epsilon$
small enough, we obtain
\begin{align}\label{eb8}
 \frac{d}{dt}||\tilde{z}^\lambda||^2+c_1||\nabla
  \tilde{z}^\lambda||^2
   \leq & K
  ||(\tilde{z}^\lambda,\tilde{E}^\lambda,\tilde{v}^\lambda)||^2+K\lambda^4||(\tilde{E}^\lambda, {\rm
  div}\tilde{E}^\lambda)||^2\nonumber\\
& +K|||\tilde{\mathbf{w}}^\lambda|||^4+K   \lambda^4.
\end{align}

 Multiplying \eqref{ea23} by $-\tilde{\Phi}^\lambda$ and integrating
  the resulting equation over $\mathbb{T}^3$ with respect to $x$,   we get
\begin{align}\label{eb9}
&\frac{\lambda^2}{2}\frac{d}{dt}||\tilde{E}^\lambda||^2
+\lambda^2||{\rm
div}\tilde{E}^\lambda||^2+\int Z|\tilde{E}^\lambda|^2 dx\nonumber\\
&\quad =-\lambda^2\int (\partial_t \mathcal{E}-\Delta
\mathcal{E})\tilde{E}^\lambda dx-\int \mathcal{E}\tilde{z}^\lambda
\tilde{E}^\lambda dx -\lambda^2\int {\rm div}\mathcal{E}
\tilde{v}^\lambda\tilde{E}^\lambda dx\nonumber\\
&\quad \ \ -\lambda^2 \int {\rm div}\mathcal{E}v \tilde{E}^\lambda
dx+\int D \tilde{v}^\lambda  \tilde{E}^\lambda dx-\lambda^2 \int v
\tilde{E}^\lambda {\rm div}\tilde{E}^\lambda dx\nonumber\\
& \quad \ \ -\lambda^2 \int \tilde{v}^\lambda  \tilde{E}^\lambda
{\rm div }\tilde{E}^\lambda dx- \int \tilde{z}^\lambda
 \tilde{E}^\lambda  \tilde{E}^\lambda dx.
\end{align}
 For the first  six terms on the
right-hand side of \eqref{eb9}, by Cauchy-Schwartz's inequality and
using the regularity of $\mathcal{E}, v $ and $D$, which can be
bounded by
\begin{equation}\label{eb10}
  \epsilon || \tilde{E}^\lambda||^2+K_\epsilon
  ||(\tilde{z}^\lambda,\tilde{v}^\lambda)||^2+
  K_\epsilon \lambda^4||( \tilde{v}^\lambda, {\rm
  div}\tilde{E}^\lambda)||^2+K_\epsilon \lambda^4.
\end{equation}
For the seventh nonlinear term, by Cauchy-Schwartz's inequality and
Sobolev's embedding $H^2(\mathbb{T}^3)\hookrightarrow
L^\infty(\mathbb{T}^3)$, we get
 \begin{align}\label{eb11}
-\lambda^2 \int \tilde{v}^\lambda  \tilde{E}^\lambda {\rm div
}\tilde{E}^\lambda dx
&\leq  \epsilon || \tilde{E}^\lambda||^2+K_\epsilon\lambda^4 || \tilde{v}^\lambda {\rm div
}\tilde{E}^\lambda||^2\nonumber\\
&\leq  \epsilon || \tilde{E}^\lambda||^2+K_\epsilon\lambda^4 || \tilde{v}^\lambda ||^2_{L^\infty}||{\rm div
}\tilde{E}^\lambda||^2\nonumber\\
&\leq  \epsilon || \tilde{E}^\lambda||^2+K_\epsilon\lambda^4 || \tilde{v}^\lambda ||^2_{H^2}||{\rm div
}\tilde{E}^\lambda||^2\nonumber\\
&\leq  \epsilon || \tilde{E}^\lambda||^2+
K_\epsilon \lambda^2|||\tilde{\mathbf{w}}^\lambda|||^4.
 \end{align}
Similarly, for the last nonlinear term, we have
\begin{align}\label{eb12}
  - \int \tilde{z}^\lambda
 \tilde{E}^\lambda  \tilde{E}^\lambda dx
 &\leq \epsilon || \tilde{E}^\lambda||^2+K_\epsilon  || \tilde{z}^\lambda  \tilde{E}^\lambda||^2\nonumber\\
&\leq  \epsilon || \tilde{E}^\lambda||^2+K_\epsilon  || \tilde{z}^\lambda ||^2_{L^\infty}||
 \tilde{E}^\lambda||^2\nonumber\\
 &\leq  \epsilon || \tilde{E}^\lambda||^2+K_\epsilon  || \tilde{z}^\lambda ||^2_{H^2}||
 \tilde{E}^\lambda||^2\nonumber\\
&\leq  \epsilon || \tilde{E}^\lambda||^2+
K_\epsilon|||\tilde{\mathbf{w}}^\lambda|||^4.
 \end{align}
Putting \eqref{eb9}-\eqref{eb12} together,   choosing $\epsilon$
small enough, and restricting $\lambda$ small enough,
   we get, by the positivity of ${Z}$, that
\begin{align}\label{eb13}
 & \lambda^2\frac{d}{dt}||\tilde{E}^\lambda||^2 + 2\lambda^2||{\rm div}\tilde{E}^\lambda||^2+c_2||\tilde{E}^\lambda||^2\nonumber\\
  &\qquad \leq K||(\tilde{z}^\lambda, \tilde{v}^\lambda)||^2 
  +K|||\tilde{\mathbf{w}}^\lambda|||^4+K\lambda^4.
\end{align}

 Multiplying \eqref{ea24} by $\tilde{v}^\lambda$ and integrating
  the resulting equation over $\mathbb{T}^3$ with respect to $x$, by \eqref{ea25} and integration  by parts, we
  obtain
  \begin{align}\label{eb14}
 & \frac{1}{2}\frac{d}{dt}||\tilde{v}^\lambda||^2+\mu ||\nabla \tilde{v}^\lambda||^2\nonumber\\
  & \quad  =-\int D \tilde{E}^\lambda\tilde{v}^\lambda dx +\lambda^2
\int \tilde{v}^\lambda\tilde{E}^\lambda {\rm div}\mathcal{E}dx
+\lambda^2\int \mathcal{E}\tilde{v}^\lambda {\rm div} \tilde{E}^\lambda dx\nonumber\\
& \quad \quad +\lambda^2 \int \tilde{v}^\lambda \mathcal{E}{\rm
div}\mathcal{E}dx -\int (\tilde{v}^\lambda \cdot \nabla
v)\tilde{v}^\lambda dx+\lambda^2\int
\tilde{v}^\lambda\tilde{E}^\lambda{\rm div}\tilde{E}^\lambda dx,
  \end{align}
where we have used the identities
$$
\int (\tilde{v}^\lambda \cdot \nabla
\tilde{v}^\lambda)\tilde{v}^\lambda dx=0,\qquad \int ({v}  \cdot
\nabla \tilde{v}^\lambda)\tilde{v}^\lambda dx=0.
$$
We estimate the terms on the
right hand side of \eqref{eb14}. For the first four terms, by
  Cauchy-Schwartz's inequality and
using the regularity of $D $ and $\mathcal{E}$, which can be
bounded by
\begin{equation}\label{eb15}
K  ||(\tilde{v}^\lambda,\tilde{E}^\lambda)||^2+
K  \lambda^2||(\tilde{v}^\lambda , \tilde{E}^\lambda, {\rm div}\tilde{E}^\lambda)||^2+K  \lambda^4.
\end{equation}
The fifth nonlinear  term can be treated as follows
\begin{equation}\label{eb16}
  \int (\tilde{v}^\lambda \cdot\nabla v)\tilde{v}^\lambda dx
  \leq K ||\nabla v||_{L^\infty}||\tilde{v}^\lambda||^2\leq K ||\tilde{v}^\lambda||^2.
\end{equation}
For the last nonlinear term, by Cauchy-Schwartz's inequality and
Sobolev's embedding $H^2(\mathbb{T}^3)\hookrightarrow
L^\infty(\mathbb{T}^3)$, we get
\begin{align}\label{eb17}
 \lambda^2\int\tilde{v}^\lambda\tilde{E}^\lambda{\rm div}\tilde{E}^\lambda dx
& \leq \frac12 ||\tilde{v}^\lambda||^2 +\frac12\lambda^4 ||\tilde{E}^\lambda{\rm div}\tilde{E}^\lambda||^2\nonumber\\
& \leq \frac12 ||\tilde{v}^\lambda||^2 +\frac12\lambda^4 ||\tilde{E}^\lambda||^2_{L^\infty}||{\rm div}\tilde{E}^\lambda||^2\nonumber\\
& \leq \frac12 ||\tilde{v}^\lambda||^2 +\frac12\lambda^4 ||\tilde{E}^\lambda||^2_{H^2}||{\rm div}\tilde{E}^\lambda||^2\nonumber\\
&\leq \frac12 ||\tilde{v}^\lambda||^2 +\frac12
|||\tilde{\mathbf{w}}^\lambda|||^4.
  \end{align}
Thus, putting \eqref{eb14}-\eqref{eb17} together, we get
\begin{align}\label{eb18}
\frac{d}{dt}||\tilde{v}^\lambda||^2+2\mu ||\nabla
\tilde{v}^\lambda||^2
\leq & K  ||(\tilde{v}^\lambda,\tilde{E}^\lambda)||^2\nonumber\\
&+ K  \lambda^2||(\tilde{v}^\lambda,\tilde{E}^\lambda, {\rm
div}\tilde{E}^\lambda)||^2
 +  |||\tilde{\mathbf{w}}^\lambda|||^4+K
 \lambda^4.
\end{align}

Combining \eqref{eb8} and  \eqref{eb13} with \eqref{eb18},  and restricting $\lambda $ small enough, we get
\begin{align}\label{eb19}
&\frac{d}{dt}\Big(\delta_1 ||\tilde{z}^\lambda||^2 +\delta_2||\tilde{v}^\lambda||^2
+\lambda^2 ||\tilde{E}^\lambda||^2\Big) +c_1\delta_1||\nabla \tilde{z}^\lambda||^2\nonumber\\
& \quad \  \ + 2\mu \delta_2||\nabla
\tilde{v}^\lambda||^2+\big(2\lambda^2-K(\lambda^2\delta_2+\lambda^4\delta_1)\big)||{\rm
div}\tilde{E}^\lambda||^2\nonumber\\
&  \quad \  \ +
\big(c_2-K(\delta_1+\delta_2)-K(\lambda^2\delta_2+\lambda^4\delta_1)\big)||\tilde{E}^\lambda||^2\nonumber\\
& \quad \leq K_1||(\tilde{z}^\lambda,
\tilde{v}^\lambda)||^2+K_1|||\tilde{\mathbf{w}}^\lambda|||^4+K_1
 \lambda^4
\end{align}
for some $\delta_1$ and $\delta_2$ sufficient small, which gives the
inequality \eqref{eb3}.
\end{proof}

Next, we estimate the $L^\infty_t(L^2_x)$ norm of $(\tilde{z_t}^\lambda,
\tilde{v_t}^\lambda, \tilde{E_t}^\lambda)$ by using the system \eqref{ea22}-\eqref{ea25}.
\begin{lem}\label{L3}
   Under the assumptions of Theorem \ref{th}, we have
   \begin{align}\label{eb20}
&||\tilde{z}^\lambda_t||^2 +||\tilde{v}^\lambda_t||^2 +\lambda^2||\tilde{E}^\lambda_t||^2 \nonumber\\
&\quad + \int^t_0 \big(||\nabla \tilde{z}^\lambda_t||^2+||\nabla\tilde{v}^\lambda_t||^2  +||\tilde{E}^\lambda_t||^2+\lambda^2
||{\rm div}\tilde{E}^\lambda_t||^2
 \big)(s)ds\nonumber\\
& \leq
 K(||\tilde{z}^\lambda_t||^2 +||\tilde{v}^\lambda_t||^2 +\lambda^2||\tilde{E}^\lambda_t||^2)(t=0)
   \nonumber\\
 & \quad+K\int^t_0\big(||(\tilde{z}^\lambda, \tilde{z}^\lambda_t,\tilde{v}^\lambda, \tilde{v}^\lambda_t,
 \nabla\tilde{v}^\lambda,
 \tilde{E}^\lambda, {\rm div} \tilde{E}^\lambda )||^2\big)(s)ds\nonumber\\
& \quad  +
 K\int^t_0 \big\{ |||\tilde{\mathbf{w}}^\lambda|||^4
+|||\tilde{\mathbf{w}}^\lambda|||^2 ||\tilde{E}^\lambda_t||^2 \big\}(s)ds+K\lambda^4.
  \end{align}
\end{lem}

\begin{proof}
  Differentiating \eqref{ea22} with respect to $t$, multiplying the resulting equation by $\tilde{z}^\lambda_t$ and
  integrating it over $\mathbb{T}^3$ with respect to $x$, we get
  \begin{align}\label{eb21}
&\frac12\frac{d}{dt}||\tilde{z}^\lambda_t||^2+||\nabla \tilde{z}^\lambda_t||^2\nonumber\\
& \quad= -\int  D \tilde{E}^\lambda_t \nabla \tilde{z}^\lambda_t dx +\lambda^2 \int \partial_t(\mathcal{E}{\rm div}\mathcal{E})
\nabla \tilde{z}^\lambda_t dx+\int \partial_t( \tilde{z}^\lambda v)\nabla \tilde{z}^\lambda_t dx\nonumber\\
&\quad  \ \ + \lambda^2\int \partial_t(\mathcal{E}{\rm div}\tilde{E}^\lambda +\tilde{E}^\lambda {\rm div}\mathcal{E})
\nabla \tilde{z}^\lambda_t dx+\int \partial_t(Z\tilde{v}^\lambda)\nabla \tilde{z}^\lambda_t dx\nonumber\\
&\quad  \ \
+\lambda^2 \int \partial_t(\tilde{E}^\lambda {\rm div}\tilde{E}^\lambda)\nabla \tilde{z}^\lambda_t dx
  +\int \partial_t( \tilde{z}^\lambda   \tilde{v}^\lambda)\nabla \tilde{z}^\lambda_t dx.
  \end{align}
We estimate the terms on the right-hand side of \eqref{eb21}. For the first five terms,
by   Cauchy-Schwartz's inequality and
using the regularity of $D,\mathcal{E}, v  $ and $Z$, which can be
bounded by
\begin{align}\label{eb22}
& \epsilon ||\nabla \tilde{z}^\lambda_t||^2 +K_\epsilon ||  \tilde{E}^\lambda_t||^2
+K_\epsilon||(\tilde{v}^\lambda, \tilde{v}^\lambda_t)||^2\nonumber\\
&\quad +K_\epsilon ||(\tilde{z}^\lambda, \tilde{z}^\lambda_t)||^2+
K_\epsilon\lambda^4||(\tilde{E}^\lambda, \tilde{E}^\lambda_t, {\rm div}\tilde{E}^\lambda, {\rm div}\tilde{E}^\lambda_t)||^2+K_\epsilon\lambda^4.
\end{align}
For the last two nonlinear terms, by Cauchy-Schwartz's inequality,
  Sobolev's embedding $H^2(\mathbb{T}^3)\hookrightarrow
L^\infty(\mathbb{T}^3)$, and the inequality \eqref{eb2},  we get
\begin{align}\label{eb23}
&   \lambda^2 \int \partial_t(\tilde{E}^\lambda {\rm div}\tilde{E}^\lambda)\nabla \tilde{z}^\lambda_t dx
 +\int \partial_t( \tilde{z}^\lambda   \tilde{v}^\lambda)\nabla \tilde{z}^\lambda_t dx \nonumber\\
\leq &  \epsilon ||\nabla \tilde{z}^\lambda_t||^2 +K_\epsilon \lambda^4 ||\partial_t(\tilde{E}^\lambda {\rm div}\tilde{E}^\lambda)||^2
+K_\epsilon||\partial_t(\tilde{z}^\lambda \tilde{v}^\lambda)||^2\nonumber\\
\leq &  \epsilon ||\nabla \tilde{z}^\lambda_t||^2
 +K_\epsilon \lambda^4 (||\tilde{E}^\lambda_t {\rm div}\tilde{E}^\lambda||^2
 +|| \tilde{E}^\lambda||^2_{L^\infty}||{\rm div}\tilde{E}^\lambda_t||^2)\nonumber\\
 &  +K_\epsilon(|| \tilde{z}^\lambda_t \tilde{v}^\lambda||^2+
 || \tilde{z}^\lambda \tilde{v}^\lambda_t||^2)\nonumber\\
\leq &  \epsilon ||\nabla \tilde{z}^\lambda_t||^2
 +K_\epsilon \lambda^4 (||\tilde{E}^\lambda_t||^2_{H^1} || {\rm div}\tilde{E}^\lambda||^2_{H^1}
 +|| \tilde{E}^\lambda||^2_{H^2}||{\rm div}\tilde{E}^\lambda_t||^2)\nonumber\\
 &  +K_\epsilon(|| \tilde{z}^\lambda_t||^2_{H^1}|| \tilde{v}^\lambda||^2_{H^1}+
 || \tilde{z}^\lambda||_{H^1}^2|| \tilde{v}^\lambda_t||^2_{H^1})\nonumber\\
\leq &  \epsilon ||\nabla \tilde{z}^\lambda_t||^2
 +K_\epsilon |||\tilde{\mathbf{w}}^\lambda|||^4.
  \end{align}
Putting \eqref{eb21}-\eqref{eb23} together and taking $
\epsilon$ small enough, we get
\begin{align}\label{eb24}
&\frac12\frac{d}{dt}||\tilde{z}^\lambda_t||^2+c_3||\nabla \tilde{z}^\lambda_t||^2\nonumber\\
& \quad \leq  K  ||  \tilde{E}^\lambda_t||^2
+K ||(\tilde{v}^\lambda, \tilde{v}^\lambda_t)||^2+K  ||(\tilde{z}^\lambda, \tilde{z}^\lambda_t)||^2\nonumber\\
&\quad\ \ \ +
K \lambda^4||(\tilde{E}^\lambda, \tilde{E}^\lambda_t, {\rm div}\tilde{E}^\lambda, {\rm div}\tilde{E}^\lambda_t)||^2
+K |||\tilde{\mathbf{w}}^\lambda|||^4+K \lambda^4.
\end{align}

Differentiating \eqref{ea23} with respect to $t$, multiplying the resulting equation by $-\tilde{\Phi}^\lambda_t$ and
  integrating it over $\mathbb{T}^3$ with respect to $x$, we get
  \begin{align}\label{eb25}
& \quad \frac{\lambda^2}{2}\frac{d}{dt} ||\tilde{E}^\lambda_t||^2
+ \lambda^2||{\rm div} \tilde{E}^\lambda_t||^2+\int Z |\tilde{E}^\lambda_t|^2 dx\nonumber\\
& \quad  =-\int Z_t \tilde{E}^\lambda \tilde{E}^\lambda_t dx-\lambda^2\int \partial_t(\partial_t \mathcal{E}-\Delta
\mathcal{E})\tilde{E}^\lambda_t dx-\int \partial_t(\mathcal{E}\tilde{z}^\lambda)
\tilde{E}^\lambda_t dx \nonumber\\
&\quad \ \ -\lambda^2\int \partial_t({\rm div}\mathcal{E}
\tilde{v}^\lambda)\tilde{E}^\lambda_t dx
 -\lambda^2 \int \partial_t(v{\rm div}\mathcal{E})
\tilde{E}^\lambda_t dx+\int \partial_t(D \tilde{v}^\lambda )
\tilde{E}^\lambda_t dx\nonumber\\
& \quad \ \ -\lambda^2 \int \partial_t(v
{\rm div}\tilde{E}^\lambda )\tilde{E}^\lambda_t dx -\lambda^2 \int \partial_t(\tilde{v}^\lambda {\rm div} \tilde{E}^\lambda)
{\rm div }\tilde{E}^\lambda_t dx- \int \partial_t(\tilde{z}^\lambda
 \tilde{E}^\lambda ) \tilde{E}^\lambda_t dx.
\end{align}
For the first seven terms on the right-hand side of \eqref{eb25}, by   Cauchy-Schwartz's inequality and
using the regularity of $\mathcal{E}, v,D $ and $Z$, which can be
bounded by
\begin{align}\label{eb26}
& \epsilon ||  \tilde{E}^\lambda_t||^2
 +K_\epsilon ||\tilde{E}^\lambda||^2+K_\epsilon ||(\tilde{z}^\lambda, \tilde{z}^\lambda_t)||^2+K_\epsilon ||(\tilde{v}^\lambda, \tilde{v}^\lambda_t)||^2
 \nonumber\\
&\qquad+K_\epsilon\lambda^4||(\tilde{v}^\lambda,
\tilde{v}^\lambda_t)||^2+ K_\epsilon\lambda^4||(  {\rm
div}\tilde{E}^\lambda, {\rm
div}\tilde{E}^\lambda_t)||^2+K_\epsilon\lambda^4.
\end{align}
For the last two nonlinear terms on the right-hand side of \eqref{eb25}  by Cauchy-Schwartz's inequality,
  Sobolev's embedding $H^2(\mathbb{T}^3)\hookrightarrow
L^\infty(\mathbb{T}^3)$, and the inequality \eqref{eb2}, they can be
estimated as follows
\begin{align}\label{eb27}
& -\lambda^2 \int \partial_t(\tilde{v}^\lambda {\rm div } \tilde{E}^\lambda)
\tilde{E}^\lambda_t dx- \int \partial_t(\tilde{z}^\lambda
 \tilde{E}^\lambda ) \tilde{E}^\lambda_t dx\nonumber\\
 & \quad \leq \epsilon ||  \tilde{E}^\lambda_t||^2 +K_\epsilon\lambda^4  ||\partial_t(\tilde{v}^\lambda  {\rm div }\tilde{E}^\lambda)||^2
 +K_\epsilon||\partial_t(\tilde{z}^\lambda
 \tilde{E}^\lambda )||^2
 \nonumber\\
& \quad \leq \epsilon ||  \tilde{E}^\lambda_t||^2
+K_\epsilon \lambda^4(||\tilde{v}^\lambda_t||^2_{H^1}
|| {\rm div } \tilde{E}^\lambda||^2_{H^1}+||\tilde{v}^\lambda ||^2_{L^\infty}|| {\rm div }\tilde{E}^\lambda_t||^2)\nonumber\\
& \quad  \ \  \ +K_\epsilon (||\tilde{z}^\lambda_t||^2_{H^1}
 ||\tilde{E}^\lambda||^2_{H^1}+||\tilde{z}^\lambda||^2_{L^\infty}
 ||\tilde{E}^\lambda_t||^2)
 \nonumber\\
 & \quad \leq \epsilon ||  \tilde{E}^\lambda_t||^2
+K_\epsilon \lambda^4(||\tilde{v}^\lambda_t||^2_{H^1}
|| {\rm div } \tilde{E}^\lambda||^2_{H^1}+||\tilde{v}^\lambda ||^2_{H^2}|| {\rm div }\tilde{E}^\lambda_t||^2)\nonumber\\
& \quad  \ \  \ +K_\epsilon (||\tilde{z}^\lambda_t||^2_{H^1}
 ||\tilde{E}^\lambda||^2_{H^1}+||\tilde{z}^\lambda||^2_{H^2}
 ||\tilde{E}^\lambda_t||^2)
 \nonumber\\
  &\quad \leq \epsilon ||  \tilde{E}^\lambda_t||^2
  +K_\epsilon\big(|||\tilde{\mathbf{w}}^\lambda|||^4+|||\tilde{\mathbf{w}}^\lambda|||^2
  ||\tilde{E}^\lambda_t||^2\big).
\end{align}
Putting \eqref{eb25}-\eqref{eb27} together, using the positivity of $Z$,  and taking $
\epsilon$ small enough, we get
\begin{align}\label{eb28}
&  {\lambda^2} \frac{d}{dt} ||\tilde{E}^\lambda_t||^2+ 2\lambda^2||{\rm div} \tilde{E}^\lambda_t||^2
+c_4||\tilde{E}^\lambda_t||^2  \nonumber\\
& \quad \leq
K||(\tilde{z}^\lambda,\tilde{z}^\lambda_t,\tilde{v}^\lambda,\tilde{v}^\lambda_t,\tilde{E}^\lambda)||^2
+K\lambda^4||({\rm div}\tilde{E}^\lambda,{\rm div}\tilde{E}^\lambda_t)||^2
+K\lambda^4||(\tilde{v}^\lambda,\tilde{v}^\lambda_t)||^2\nonumber\\
&\quad \ \ +K\big( |||\tilde{\mathbf{w}}^\lambda|||^4
+|||\tilde{\mathbf{w}}^\lambda|||^2||\tilde{E}^\lambda_t||^2 \big)+K\lambda^4.
\end{align}

Differentiating \eqref{ea24} with respect to $t$, multiplying the resulting equation by $\tilde{v}^\lambda_t$,
  integrating it over $\mathbb{T}^3$ with respect to $x$ and using $\text{div} \tilde{v}^\lambda_t=0$, we get
  \begin{align}\label{eb29}
  &\frac12\frac{d}{dt}||\tilde{v}^\lambda_t||^2 +\mu ||\nabla \tilde{v}^\lambda_t||^2\nonumber\\
  & \quad = -\int \partial_t(D \tilde{E}^\lambda)\tilde{v}^\lambda_t dx
  +\lambda^2 \int \partial_t (\mathcal{E}{\rm div}\mathcal{E})\tilde{v}^\lambda_t dx
+\lambda^2 \int \partial_t (\mathcal{E}{\rm div}\tilde{E}^\lambda)\tilde{v}^\lambda_t dx\nonumber\\
& \quad\  \ \ +\lambda^2 \int \partial_t (\tilde{E}{\rm div}\mathcal{E})\tilde{v}^\lambda_t dx
  -\int \partial_t( {v}\cdot \nabla \tilde{v}^\lambda)\tilde{v}^\lambda_t dx
  -\int \partial_t(\tilde{v}^\lambda\cdot \nabla  {v} )\tilde{v}^\lambda_t dx\nonumber\\
  & \quad\  \ \ -\int \partial_t(\tilde{v}^\lambda\cdot \nabla \tilde{v}^\lambda)\tilde{v}^\lambda_t dx
+\lambda^2 \int \partial_t (\tilde{E}^\lambda{\rm div}\tilde{E}^\lambda)\tilde{v}^\lambda_t dx.
   \end{align}
We estimate the terms on the right-hand side of \eqref{eb29}. By  Cauchy-Schwartz's inequality and
using the regularity of $D$ and $\mathcal{E}$,  we get
\begin{align}\label{eb31}
& -\int \partial_t(D \tilde{E}^\lambda)\tilde{v}^\lambda_t dx
  +\lambda^2 \int \partial_t (\mathcal{E}{\rm div}\mathcal{E})\tilde{v}^\lambda_t dx\nonumber\\
&+\lambda^2 \int \partial_t (\mathcal{E}{\rm div}\tilde{E}^\lambda)\tilde{v}^\lambda_t dx
+\lambda^2 \int \partial_t (\tilde{E}{\rm div}\mathcal{E})\tilde{v}^\lambda_t dx\nonumber\\
& \quad \leq K(||\tilde{v}^\lambda_t||^2 +||\tilde{E}^\lambda||^2+||\tilde{E}^\lambda_t||^2)
+K \lambda^4 ||({\rm div}\tilde{E}^\lambda, {\rm div}\tilde{E}^\lambda_t)||^2\nonumber\\
& \quad \ \
+K\lambda^4||(\tilde{E}^\lambda, \tilde{E}^\lambda_t)||^2+K\lambda^4.
\end{align}
Now we deal with the trilinear terms involving $\tilde{v}^\lambda,\tilde{v}^\lambda_t$, ${v}$, and ${v}_t$.
Using the
  identities
$$
\int (\tilde{v}^\lambda \cdot \nabla
\tilde{v}^\lambda_t)\tilde{v}^\lambda_t dx=0,\quad \int ({v}  \cdot
\nabla \tilde{v}^\lambda_t)\tilde{v}^\lambda_t dx=0,
$$
we have
\begin{align}\label{eb300}
&\quad  -\int \partial_t(\tilde{v}^\lambda\cdot \nabla \tilde{v}^\lambda)\tilde{v}^\lambda_t dx
  -\int \partial_t( {v}\cdot \nabla \tilde{v}^\lambda)\tilde{v}^\lambda_t dx
  -\int \partial_t(\tilde{v}^\lambda\cdot \nabla  {v} )\tilde{v}^\lambda_t dx\nonumber\\
& =    -\int (\tilde{v}^\lambda_t\cdot \nabla \tilde{v}^\lambda)\tilde{v}^\lambda_t dx
  -\int ( {v}_t\cdot \nabla \tilde{v}^\lambda)\tilde{v}^\lambda_t dx
  -\int (\tilde{v}^\lambda_t\cdot \nabla  {v} )\tilde{v}^\lambda_t dx
  -\int (\tilde{v}^\lambda\cdot \nabla  {v}_t )\tilde{v}^\lambda_t dx.
\end{align}
By  Cauchy-Schwartz's inequality, using the regularity of $ {v}$
and the inequality \eqref{eb2},  we get
\begin{align}
   -\int (\tilde{v}^\lambda_t\cdot \nabla \tilde{v}^\lambda)\tilde{v}^\lambda_t dx
   &\leq \frac12 || \tilde{v}^\lambda_t||^2
    +\frac{ 1}{2} ||\tilde{v}^\lambda_t\cdot \nabla \tilde{v}^\lambda||^2\nonumber\\
   & \leq \frac12 || \tilde{v}^\lambda_t||^2+
    \frac{ K}{2} ||\tilde{v}^\lambda_t||^2_{H^1}|| \nabla \tilde{v}^\lambda||^2_{H^1}\nonumber\\
&  \leq \frac12 || \tilde{v}^\lambda_t||^2+K |||\tilde{\mathbf{w}}^\lambda|||^4,\label{eb301}\\
  -\int ( {v}_t\cdot \nabla \tilde{v}^\lambda)\tilde{v}^\lambda_t dx
   &\leq \frac12 || \tilde{v}^\lambda_t||^2
    +\frac{ 1}{2} || {v}_t\cdot \nabla \tilde{v}^\lambda||^2\nonumber\\
 &\leq   \frac12 || \tilde{v}^\lambda_t||^2+K|| \nabla \tilde{v}^\lambda||^2, \label{eb302}\\
    -\int (\tilde{v}^\lambda_t\cdot \nabla  {v} )\tilde{v}^\lambda_t dx
  &  \leq K ||\nabla v||_{L^\infty}||\tilde{v}^\lambda_t||^2\leq K ||\tilde{v}^\lambda_t||^2,\label{eb303}\\
  -\int (\tilde{v}^\lambda\cdot \nabla  {v}_t )\tilde{v}^\lambda_t dx
 &\leq \frac12 || \tilde{v}^\lambda_t||^2
    +\frac{ 1}{2} ||\tilde{v}^\lambda\cdot \nabla {v}_t||^2\nonumber\\
&  \leq \frac12 || \tilde{v}^\lambda_t||^2+K ||\tilde{v}^\lambda
||^2.\label{eb304}
\end{align}
The last nonlinear term can be treated similarly as \eqref{eb23}
\begin{align}\label{eb32}
  \lambda^2 \int \partial_t (\tilde{E}^\lambda{\rm div}\tilde{E}^\lambda)\tilde{v}^\lambda_t dx
\leq & \frac12 ||\tilde{v}^\lambda_t||^2+\frac12 \lambda^4 ||\partial_t(\tilde{E}^\lambda {\rm div}\tilde{E}^\lambda)||^2\nonumber\\
\leq & \frac12 ||\tilde{v}^\lambda_t||^2+\frac12 \lambda^4 (||\tilde{E}^\lambda_t {\rm div}\tilde{E}^\lambda||^2
 +|| \tilde{E}^\lambda||^2_{L^\infty}||{\rm div}\tilde{E}^\lambda_t||^2)\nonumber\\
\leq &  \frac12 ||\tilde{v}^\lambda_t||^2+\frac12 \lambda^4 (||\tilde{E}^\lambda_t||^2_{H^1} || {\rm div}\tilde{E}^\lambda||^2_{H^1}
 +|| \tilde{E}^\lambda||^2_{H^2}||{\rm div}\tilde{E}^\lambda_t||^2)\nonumber\\
\leq & \frac12 ||\tilde{v}^\lambda_t||^2+K|||\tilde{\mathbf{w}}^\lambda|||^4.
\end{align}
Putting \eqref{eb29}-\eqref{eb32} together, we have
\begin{align}\label{eb33}
& \frac{d}{dt}||\tilde{v}^\lambda_t||^2+\mu ||\nabla \tilde{v}^\lambda_t||^2\nonumber\\
&\quad \leq
K||(\tilde{v}^\lambda,\nabla\tilde{v}^\lambda,\tilde{v}^\lambda_t,\tilde{E}^\lambda,\tilde{E}^\lambda_t)||^2
+K\lambda^4||(\tilde{E}^\lambda, \tilde{E}^\lambda_t, {\rm div}\tilde{E}^\lambda, {\rm div}\tilde{E}^\lambda_t)||^2\nonumber\\
& \quad \ \ +K|||\tilde{\mathbf{w}}^\lambda|||^4+K\lambda^4.
\end{align}

Combining \eqref{eb24}, \eqref{eb28} and \eqref{eb33},  and restricting $\lambda $ small enough, we get
\begin{align}\label{eb34}
 & \frac{d}{dt}(\delta_3||\tilde{z}^\lambda_t||^2+ \lambda^2||\tilde{E}^\lambda_t||^2+\delta_4||\tilde{v}^\lambda_t||^2)\nonumber\\
& \ \  + 2\delta_3c_3 ||\nabla\tilde{z}^\lambda_t||^2+\mu
\delta_4||\nabla
\tilde{v}^\lambda_t||^2+c_5||\tilde{E}^\lambda_t||^2
+c_6\lambda^2||{\rm div}\tilde{E}^\lambda_t||^2\nonumber\\
&\quad \leq K_2||(\tilde{v}^\lambda,\nabla\tilde{v}^\lambda,
\tilde{v}^\lambda_t,\tilde{z}^\lambda,\tilde{z}^\lambda_t,\tilde{E}^\lambda,
{\rm div}\tilde{E}^\lambda)||^2\nonumber\\
&\quad \ \
  +K_2\big( |||\tilde{\mathbf{w}}^\lambda|||^4
+|||\tilde{\mathbf{w}}^\lambda|||^2||\tilde{E}^\lambda_t||^2 \big)+K_2\lambda^4.
\end{align}
for some $\delta_3$ and $\delta_4$ sufficient small, which gives the inequality \eqref{eb20}.
 \end{proof}

Using Lemma \ref{L3}, we can obtain the $L^\infty_t(L^2_x)$ norm of
 $(\nabla \tilde{z}^\lambda,\nabla \tilde{v}^\lambda, \tilde{E}^\lambda,\lambda{\rm div}\tilde{E}^\lambda)$.
\begin{lem}\label{L4}
Under the assumptions of Theorem \ref{th}, we have
   \begin{align}\label{eb35}
& ||(\nabla \tilde{z}^\lambda,\nabla \tilde{v}^\lambda, \tilde{E}^\lambda)||^2
+\lambda^2||{\rm div}\tilde{E}^\lambda||^2\nonumber\\
& \quad \leq K||(\tilde{z}^\lambda,\tilde{z}^\lambda_t,\tilde{v}^\lambda,\tilde{v}^\lambda_t)||^2
+K\lambda^2||\tilde{E}^\lambda_t||^2+K|||\tilde{\mathbf{w}}^\lambda|||^4+K\lambda^4.
\end{align}
\end{lem}

\begin{proof}
  It follows form \eqref{eb19} and Cauchy-Schwartz's inequality that
\begin{align*}
&  \   c_1\delta_1||\nabla \tilde{z}^\lambda||^2
  + \mu \delta_2||\nabla
\tilde{v}^\lambda||^2+\big(2\lambda^2-K(\lambda^2\delta_2+\lambda^4\delta_1)\big)||{\rm
div}\tilde{E}^\lambda||^2\nonumber\\
&+
\big(c_2-K(\delta_1+\delta_2)-K(\lambda^2\delta_2+\lambda^4\delta_1)\big)||\tilde{E}^\lambda||^2\nonumber\\
& \quad \leq -\frac{d}{dt}\Big(\delta_1 ||\tilde{z}^\lambda||^2 +\delta_2||\tilde{v}^\lambda||^2
+\lambda^2 ||\tilde{E}^\lambda||^2\Big)\\
& \quad\quad + K_1||(\tilde{z}^\lambda,
\tilde{v}^\lambda)||^2+K_1|||\tilde{\mathbf{w}}^\lambda|||^4+K_1
 \lambda^4\\
  & \quad \leq  K||(\tilde{z}^\lambda, \tilde{z}^\lambda_t,\tilde{v}^\lambda, \tilde{v}^\lambda_t)||^2+
 K\lambda^2||(\tilde{E}^\lambda, \tilde{E}^\lambda_t)||^2+
 K|||\tilde{\mathbf{w}}^\lambda|||^4+K
 \lambda^4,
\end{align*}
which gives \eqref{eb35} by using Lemma \ref{L3}.
\end{proof}

\subsection{High order estimates}
In this subsection we will establish the $L^\infty_t(L^2_x)$ of the
higher order spatial derivatives $(\Delta \tilde{z}^\lambda,\Delta
\tilde{v}^\lambda, \text{div}   \tilde{E}^\lambda, \lambda \nabla
\text{div}    \tilde{E}^\lambda)$.
\begin{lem}\label{L5}
Under the assumptions of Theorem \ref{th}, we have
\begin{align}\label{eb36}
 & ||\Delta \tilde{z}^\lambda ||^2+||\Delta \tilde{v}^\lambda ||^2+  ||{\rm div}\tilde{E}^\lambda||^2
 + \lambda^2 ||\nabla {\rm div}\tilde{E}^\lambda||^2\nonumber\\
 \leq &  K  ||(\tilde{z}^\lambda,\nabla\tilde{z}^\lambda, \nabla\tilde{z}^\lambda_t,
 \tilde{v}^\lambda,\nabla\tilde{v}^\lambda, \nabla\tilde{v}^\lambda_t,  \tilde{E}^\lambda)||^2\nonumber\\
 & +K\lambda^2 ||{\rm div} \tilde{E}^\lambda_t||^2+K|||\tilde{\mathbf{w}}^\lambda|||^4+K\lambda^4.
  \end{align}
\end{lem}

\begin{proof}
 Multiplying \eqref{ea22} by $-\Delta \tilde{z}^\lambda$, integrating
  the resulting equation over $\mathbb{T}^3$ with respect to $x$, we get
\begin{align}\label{eb37}
&\frac{1}{2}\frac{d}{dt}||\nabla\tilde{z}^\lambda||^2 +||\Delta \tilde{z}^\lambda||^2\nonumber\\
&\ \ =-\int {\rm div}(D\tilde{E}^\lambda) \Delta \tilde{z}^\lambda dx+ \lambda^2\int{\rm div} (\tilde{E}^\lambda{\rm
  div}\mathcal{E}+\mathcal{E}{\rm
  div}\tilde{E}^\lambda)\Delta \tilde{z}^\lambda dx\nonumber\\
& \quad \ \ \ +
\lambda^2\int {\rm div}(\mathcal{E}{\rm div} \mathcal{E})\Delta \tilde{z}^\lambda dx
+\int {\rm div}(Z \tilde{v}^\lambda)\Delta \tilde{z}^\lambda dx
  +\int {\rm div}(v \tilde{z}^\lambda)\Delta \tilde{z}^\lambda dx\nonumber\\
 &\quad \ \  \ +\int {\rm div}(\tilde{z}^\lambda \tilde{v}^\lambda)\Delta \tilde{z}^\lambda dx
  +\lambda^2\int {\rm div}(\tilde{E}^\lambda {\rm div}\tilde{E}^\lambda)\Delta \tilde{z}^\lambda dx.
\end{align}
We estimate the terms on the right-hand side of \eqref{eb37}. By  Cauchy-Schwartz's inequality and
using the regularity of $D,\mathcal{E}, Z $ and $v$,  the first two terms can be bounded by
\begin{equation}\label{eb38}
\epsilon ||\Delta \tilde{z}^\lambda||^2+K_\epsilon||(\tilde{E}^\lambda,{\rm div}\tilde{E}^\lambda)||^2
+K_\epsilon\lambda^4||(\tilde{E}^\lambda,{\rm div}\tilde{E}^\lambda,\nabla{\rm div}\tilde{E}^\lambda)||^2
\end{equation}
and the third, fourth and fifth terms can be bounded by
\begin{equation}\label{eb39}
\epsilon ||\Delta \tilde{z}^\lambda||^2+K_\epsilon||(\nabla \tilde{z}^\lambda,\tilde{v}^\lambda)
 ||^2+K_\epsilon \lambda^4,
\end{equation}
where we use the facts that $\text{div}\tilde{v}^\lambda=0$ and $\text{div} {v} =0$.
For the last two  nonlinear terms, using the facts that
$$
 {\rm div} (\tilde{z}^\lambda\tilde{v}^\lambda)=\nabla
\tilde{z}^\lambda\tilde{v}^\lambda,\quad {\rm div}(\tilde{E}{\rm
div}\tilde{E}^\lambda)=\tilde{E}^\lambda \nabla {\rm
div}\tilde{E}^\lambda+({\rm div}\tilde{E}^\lambda)^2,
$$
  Cauchy-Schwartz's inequality,   Sobolev's embedding
$H^2(\mathbb{T}^3)\hookrightarrow L^\infty(\mathbb{T}^3)$ and the
inequality \eqref{eb2}, we have
\begin{align}\label{eb40}
&\int {\rm div}(\tilde{z}^\lambda \tilde{v}^\lambda)\Delta \tilde{z}^\lambda dx
  +\lambda^2\int {\rm div}(\tilde{E}^\lambda {\rm div}\tilde{E}^\lambda)\Delta \tilde{z}^\lambda dx\nonumber\\
  &\quad \leq \epsilon ||\Delta \tilde{z}^\lambda||^2+K_\epsilon ||\nabla \tilde{z}^\lambda \tilde{v}^\lambda||^2
 + K_\epsilon \lambda^4 ||{\rm div}( \tilde{E}^\lambda {\rm div}\tilde{E}^\lambda)||^2\nonumber\\
  & \quad \leq    \epsilon ||\Delta \tilde{z}^\lambda||^2+K_\epsilon ||\nabla \tilde{z}^\lambda||^2|| \tilde{v}^\lambda||^2_{L^\infty}\nonumber\\
 & \quad \ \ \   +K_\epsilon \lambda^4 (|| \tilde{E}^\lambda||^2_{L^\infty}||\nabla {\rm div}\tilde{E}^\lambda||^2+
||{\rm div}  \tilde{E}^\lambda ||^2_{H^1}||{\rm div}\tilde{E}^\lambda||^2_{H^1})\nonumber\\
  & \quad \leq    \epsilon ||\Delta \tilde{z}^\lambda||^2+K_\epsilon ||\nabla \tilde{z}^\lambda||^2|| \tilde{v}^\lambda||^2_{H^2}\nonumber\\
 & \quad \ \ \   +K_\epsilon \lambda^4 (|| \tilde{E}^\lambda||^2_{H^2}||\nabla {\rm div}\tilde{E}^\lambda||^2+
||{\rm div}  \tilde{E}^\lambda ||^2_{H^1}||{\rm div}\tilde{E}^\lambda||^2_{H^1})\nonumber\\
 & \quad \leq \epsilon ||\Delta \tilde{z}^\lambda||^2+K_\epsilon |||\tilde{\mathbf{w}}^\lambda|||^4.
\end{align}
Putting \eqref{eb37}-\eqref{eb40} together and choosing $\epsilon$ small enough, we have
\begin{align}\label{eb41}
&\frac{d}{dt}||\nabla\tilde{z}^\lambda||^2 +c_7||\Delta \tilde{z}^\lambda||^2\nonumber\\
&\quad \leq K||(\tilde{E}^\lambda, {\rm div}\tilde{E}^\lambda)||^2
+K\lambda^4||(\tilde{E}^\lambda,{\rm div}\tilde{E}^\lambda,\nabla {\rm div}\tilde{E}^\lambda)||^2\nonumber\\
& \quad  \ \ \  +K ||(\nabla \tilde{z}^\lambda, \tilde{v}^\lambda)||^2+K |||\tilde{\mathbf{w}}^\lambda|||^4+K \lambda^4.
\end{align}

Multiplying \eqref{ea23} by ${\rm div}\tilde{E}^\lambda$ and integrating
  the resulting equation over $\mathbb{T}^3$ with respect to $x$, we get
\begin{align}\label{eb42}
& \frac{\lambda^2}{2}\frac{d}{dt}||{\rm div}\tilde{E}^\lambda||^2+\lambda^2||\nabla {\rm div}\tilde{E}^\lambda||^2
+\int Z |{\rm div}\tilde{E}^\lambda|^2 dx\nonumber\\
&   =  - \int \nabla Z
\tilde{E}^\lambda {\rm div}\tilde{E}^\lambda  dx -\lambda^2\int {\rm div}(\partial_t \mathcal{E}-\Delta
\mathcal{E}){\rm div}\tilde{E}^\lambda  dx-\int {\rm div}(\mathcal{E}\tilde{z}^\lambda)
{\rm div}\tilde{E}^\lambda dx \nonumber\\
&  \ \ \ -\lambda^2 \int  \tilde{v}^\lambda \nabla{\rm div}\mathcal{E}  {\rm div}\tilde{E}^\lambda
dx -\lambda^2 \int v  \nabla{\rm div}\mathcal{E}   {\rm div}\tilde{E}^\lambda dx
+\int \tilde{v}^\lambda \nabla D {\rm div}\tilde{E}^\lambda dx
\nonumber\\
&  \ \ \
-\lambda^2 \int v \nabla {\rm div}\tilde{E}^\lambda   {\rm div}\tilde{E}^\lambda dx
-\lambda^2 \int \tilde{v}^\lambda \nabla {\rm div }\tilde{E}^\lambda {\rm div}\tilde{E}^\lambda dx
-\int {\rm div}(\tilde{z}^\lambda   \tilde{E}^\lambda){\rm div}\tilde{E}^\lambda dx,
\end{align}
where we have used $\text{div}\tilde{v}^\lambda =0$ and $\text{div}{v}  =0$.
 By  Cauchy-Schwartz's inequality and
using the regularity of $\mathcal{E}, v,D $ and $Z$,  the first seven terms on the right hand side of \eqref{eb42}
can be bounded by
\begin{equation}\label{eb43}
\epsilon||{\rm div}\tilde{E}^\lambda||^2
+K_\epsilon ||(\tilde{E}^\lambda,\tilde{z}^\lambda, \nabla \tilde{z}^\lambda,\tilde{v}^\lambda)||^2
+K_\epsilon \lambda^4||(\tilde{v}^\lambda,\nabla {\rm div}\tilde{E}^\lambda)||^2 +K_\epsilon \lambda^4.
\end{equation}
By  Cauchy-Schwartz's inequality, Sobolev's embedding $H^2(\mathbb{T}^3)\hookrightarrow
L^\infty(\mathbb{T}^3)$, and
using the inequality \eqref{eb2}, the last two nonlinear terms on the right hand side of \eqref{eb42}
can be treated as follows
\begin{align}\label{eb44}
&   -\lambda^2 \int \tilde{v}^\lambda \nabla {\rm div }\tilde{E}^\lambda {\rm div}\tilde{E}^\lambda dx
-\int {\rm div}(\tilde{z}^\lambda   \tilde{E}^\lambda){\rm div}\tilde{E}^\lambda dx \nonumber\\
\leq & \epsilon||{\rm div}\tilde{E}^\lambda||^2 +K_\epsilon \lambda^4 ||\tilde{v}^\lambda \nabla {\rm div }\tilde{E}^\lambda||^2
+K_\epsilon||{\rm div}(\tilde{z}^\lambda   \tilde{E}^\lambda)||^2\nonumber\\
\leq & \epsilon||{\rm div}\tilde{E}^\lambda||^2
+K_\epsilon \lambda^4  ||\tilde{v}^\lambda||^2_{L^\infty}|| \nabla {\rm div }\tilde{E}^\lambda||^2 \nonumber\\
  & +K_\epsilon(|| \tilde{z}^\lambda ||^2_{L^\infty}||{\rm div } \tilde{E}^\lambda||^2
+||{\nabla}\tilde{z}^\lambda||^2_{H^1}|| \tilde{E}^\lambda||^2_{H^1})\nonumber\\
\leq & \epsilon||{\rm div}\tilde{E}^\lambda||^2
+K_\epsilon \lambda^4  ||\tilde{v}^\lambda||^2_{H^2}|| \nabla {\rm div }\tilde{E}^\lambda||^2 \nonumber\\
  & +K_\epsilon(|| \tilde{z}^\lambda ||^2_{H^2}||{\rm div } \tilde{E}^\lambda||^2
+||{\nabla}\tilde{z}^\lambda||^2_{H^1}|| \tilde{E}^\lambda||^2_{H^1})\nonumber\\
\leq & \epsilon||{\rm div}\tilde{E}^\lambda||^2
 +K_\epsilon (1+\lambda^2)|||\tilde{\mathbf{w}}^\lambda|||^4.
  \end{align}
Putting \eqref{eb42}-\eqref{eb44} together and using the positivity of ${Z}$, we get
\begin{align}\label{eb45}
&\lambda^2 \frac{d}{dt}||{\rm div}\tilde{E}^\lambda||^2
+2\lambda^2 ||\nabla {\rm div}\tilde{E}^\lambda||^2+c_{8}||{\rm div}\tilde{E}^\lambda||^2\nonumber\\
&\quad \leq K||(\tilde{z}^\lambda, \nabla \tilde{z}^\lambda,\tilde{E}^\lambda,\tilde{v}^\lambda)||^2
+K  \lambda^4||(\tilde{v}^\lambda,\nabla {\rm div}\tilde{E}^\lambda)||^2\nonumber\\
 &\quad \ \ \ +K (1+\lambda^2)|||\tilde{\mathbf{w}}^\lambda|||^4+K  \lambda^4.
\end{align}

Multiplying \eqref{ea24} by $-\Delta\tilde{v}^\lambda$ and integrating
  the resulting equation over $\mathbb{T}^3$ with respect to $x$, by \eqref{ea25} and integrating it by parts, we have
  \begin{align}\label{eb46}
& \frac{1}{2}\frac{d}{dt}||\nabla \tilde{v}^\lambda||^2 +\mu ||\Delta \tilde{v}^\lambda||^2\nonumber\\
& \quad = \int (v \cdot \nabla \tilde{v}^\lambda) \Delta\tilde{v}^\lambda dx
+\int (\tilde{v}^\lambda\cdot \nabla  {v})  \Delta\tilde{v}^\lambda dx
- \lambda^2\int \tilde{E}^\lambda\text{div}\mathcal{E}  \Delta\tilde{v}^\lambda dx\nonumber\\
& \quad \ \ \
 -\lambda^2\int \mathcal{E}\text{div}\tilde{E}^\lambda  \Delta\tilde{v}^\lambda dx
-\lambda^2\int \mathcal{E}\text{div}\mathcal{E} \Delta\tilde{v}^\lambda dx
+\int D\tilde{E}^\lambda  \Delta\tilde{v}^\lambda dx \nonumber\\
&\quad \ \ \ -\lambda^2\int \tilde{E}^\lambda\text{div}\tilde{E}^\lambda  \Delta\tilde{v}^\lambda dx
+\int(\tilde{v}^\lambda\cdot \nabla  \tilde{v}^\lambda ) \Delta\tilde{v}^\lambda dx
 \end{align}
We estimate the terms on the right-hand side of \eqref{eb46}. By   Cauchy-Schwartz's inequality and
using the regularity of $v,\mathcal{E} $ and $D$,  the first six terms
  can be bounded by
\begin{equation}\label{eb47}
\epsilon||\Delta \tilde{v}^\lambda||^2
+K_\epsilon ||(\tilde{v}^\lambda,\nabla \tilde{v}^\lambda,\tilde{E}^\lambda)||^2
+K_\epsilon \lambda^4||(\tilde{E}^\lambda,  {\rm div}\tilde{E}^\lambda)||^2 +K_\epsilon \lambda^4.
\end{equation}
By  Cauchy-Schwartz's inequality, Sobolev's embedding
$H^2(\mathbb{T}^3)\hookrightarrow L^\infty(\mathbb{T}^3)$, and using
the inequality \eqref{eb2}, the last two nonlinear terms   can be
treated as follows
\begin{align}\label{eb48}
&\qquad-\lambda^2\int \tilde{E}^\lambda\text{div}\tilde{E}^\lambda  \Delta\tilde{v}^\lambda dx
+\int(\tilde{v}^\lambda\cdot \nabla  \tilde{v}^\lambda)  \Delta\tilde{v}^\lambda dx \nonumber\\
& \quad \leq \epsilon||\Delta \tilde{v}^\lambda||^2 +K_\epsilon ||\tilde{v}^\lambda\cdot \nabla  \tilde{v}^\lambda ||^2
+K_\epsilon\lambda^4 ||\tilde{E}^\lambda\text{div}\tilde{E}^\lambda||^2\nonumber\\
& \quad \leq \epsilon||\Delta \tilde{v}^\lambda||^2
+K_\epsilon ||\tilde{v}^\lambda||^2_{L^\infty}||\nabla  \tilde{v}^\lambda ||^2
+K_\epsilon\lambda^4 ||\tilde{E}^\lambda||^2_{L^\infty}||\text{div}\tilde{E}^\lambda||^2\nonumber\\
& \quad \leq \epsilon||\Delta \tilde{v}^\lambda||^2
+K_\epsilon ||\tilde{v}^\lambda||^2_{H^2}||\nabla  \tilde{v}^\lambda ||^2
+K_\epsilon\lambda^4 ||\tilde{E}^\lambda||^2_{H^2}||\text{div}\tilde{E}^\lambda||^2\nonumber\\
&\quad  \leq \epsilon||\Delta \tilde{v}^\lambda||^2+ K_\epsilon |||\tilde{\mathbf{w}}^\lambda|||^4.
 \end{align}
Putting \eqref{eb46}-\eqref{eb48} together and choosing $\epsilon$ small enough, we have
\begin{align}\label{eb49}
\frac{d}{dt}||\nabla \tilde{v}^\lambda||^2 +c_{10}\mu ||\Delta
\tilde{v}^\lambda||^2
&\leq K||(\tilde{v}^\lambda,\nabla \tilde{v}^\lambda, \tilde{E}^\lambda)||^2\nonumber\\
&\  \
+K\lambda^4||(\tilde{E}^\lambda, {\rm div}\tilde{E}^\lambda)||^2
+K |||\tilde{\mathbf{w}}^\lambda|||^4+K\lambda^4.
\end{align}

Combining \eqref{eb41}, \eqref{eb45} and \eqref{eb49},   and restricting $\lambda$ is small, we get
\begin{align}\label{eb50}
&\frac{d}{dt}\Big(\delta_5 ||\nabla \tilde{z}^\lambda||^2+\lambda^2||{\rm div}\tilde{E}^\lambda||^2+\delta_6||\nabla \tilde{v}^\lambda||^2\Big)\nonumber\\
&\quad \  \ + \delta_5c_8||\Delta \tilde{z}^\lambda||^2+
(2\lambda^2-K\lambda^4\delta_5-K\lambda^4)||\nabla {\rm div} \tilde{E}^\lambda||^2+\delta_6\mu c_{9}||\Delta \tilde{v}^\lambda||^2\nonumber\\
&\quad \  \ +(c_{8}-K\delta_5(\lambda^4+1)-K\delta_6\lambda^4)||{\rm div} \tilde{E}^\lambda||^2\nonumber\\
& \quad \leq K_3 ||( \tilde{E}^\lambda,\tilde{z}^\lambda  ,\nabla \tilde{z}^\lambda,\tilde{v}^\lambda,\nabla \tilde{v}^\lambda)||^2
+K_3 |||\tilde{\mathbf{w}}^\lambda|||^4+K_3\lambda^4
\end{align}
for some $\delta_5$ and $\delta_6$ sufficient small, which gives the   inequality \eqref{eb36}.
\end{proof}

In order to close the estimates on the right-hand side of \eqref{eb36}, we need to obtain the uniform bounds of the
time derivatives $(  \nabla \tilde{z}^\lambda_t,
  \nabla \tilde{v}^\lambda_t,\lambda  \text{div}\tilde{E}^\lambda_t)$, which
is given by the next lemma.
\begin{lem}\label{L6}
Under the assumptions of Theorem \ref{th}, we have
\begin{align}\label{eb51}
&\quad ||\nabla \tilde{z}^\lambda_t||^2+\lambda^2||{\rm div}\tilde{E}^\lambda_t||^2
+ ||\nabla \tilde{v}^\lambda_t||^2 \nonumber\\
&\quad +\int^t_0 \big(||\Delta \tilde{z}^\lambda_t||^2+||\Delta \tilde{v}^\lambda_t||^2
      +||{\rm div}\tilde{E}^\lambda_t||^2+\lambda^2 ||\nabla {\rm div}\tilde{E}^\lambda_t||^2
\big)(s)dx\nonumber\\
&  \leq K\big(||\nabla \tilde{z}^\lambda_t||^2+\lambda^2||{\rm div}\tilde{E}^\lambda_t||^2
+ ||\nabla \tilde{v}^\lambda_t||^2\big)(t=0)\nonumber\\
&\quad +K \int^t_0 \big(||(\tilde{z}^\lambda, \tilde{z}^\lambda_t, \nabla \tilde{z}^\lambda,\nabla \tilde{z}^\lambda_t)||^2+
  ||(\tilde{v}^\lambda, \tilde{v}^\lambda_t, \nabla \tilde{v}^\lambda,\nabla \tilde{v}^\lambda_t)||^2)(s)ds\nonumber\\
&\quad     +K  \int^t_0 \big( ||(\tilde{E}^\lambda, \tilde{E}^\lambda_t, {\rm div} \tilde{E}^\lambda,\nabla {\rm div}\tilde{E}^\lambda)||^2)(s)ds
+K \int^t_0   |||\tilde{\mathbf{w}}^\lambda|||^4(s)ds\nonumber\\
&\quad
+K  \int^t_0 \Big\{|||\tilde{\mathbf{w}}^\lambda|||^2
\big(
|| \tilde{z}^\lambda_t||^2_{H^2}+||\nabla\tilde{z}^\lambda_t||^2_{H^1}+ || \tilde{E}^\lambda_t||^2_{H^1}
+||{\rm div}\tilde{E}^\lambda_t||^2\big)\Big\}(s)ds\nonumber\\
&\quad    +K \lambda^2\int^t_0 \Big\{ |||\tilde{\mathbf{w}}^\lambda|||^2
\big(
|| \tilde{E}^\lambda_t||^2_{H^2}+||{\rm div}\tilde{E}^\lambda_t||^2_{H^1}+
||\nabla{\rm div}\tilde{E}^\lambda_t||^2
 \big)\Big\}(s)ds+K \lambda^4,
\end{align}
\end{lem}

\begin{proof}
Differentiating \eqref{ea22} with respect to $t$, multiplying the resulting equation by $-\Delta \tilde{z}^\lambda_t$ and
  integrating it over $\mathbb{T}^3$ with respect to $x$, we get
  \begin{align}\label{eb52}
   & \frac12 \frac{d}{dt}||\nabla \tilde{z}^\lambda_t||^2+||\Delta\tilde{z}^\lambda_t||^2\nonumber\\
&=\int \Big\{-{\rm div}(D\tilde{E}^\lambda_t)+\lambda^2\partial_t[\text{div}(\mathcal{E}\text{div}\tilde{E}^\lambda+\tilde{E}^\lambda
\text{div}\mathcal{E})]+\lambda^2\partial_t[\text{div}(\mathcal{E}\text{div}\mathcal{E})]\nonumber\\
&  \quad +\partial_t\text{div}(\tilde{z}^\lambda v)+\partial_t[\text{div}(Z\tilde{v}^\lambda)]\Big\}\Delta\tilde{z}^\lambda_tdx\nonumber\\
&\quad +\int\Big\{ \partial_t[\text{div}(\tilde{z}^\lambda\tilde{v}^\lambda)]
+\lambda^2\partial_t[\text{div}(\tilde{E}^\lambda\text{div}\tilde{E}^\lambda)]\Big\}\Delta\tilde{z}^\lambda_tdx.
  \end{align}
We estimate the terms on the right-hand side of \eqref{eb52}.
By   Cauchy-Schwartz's inequality and using the regularity of
$D, \mathcal{E}, v  $ and $Z$,  the first integral can be bounded by
\begin{align}\label{eb53}
&\epsilon||\Delta \tilde{z}^\lambda_t||^2+K_\epsilon ||(\tilde{E}^\lambda_t, {\rm div}\tilde{E}^\lambda_t)||^2
+K_\epsilon||(\nabla\tilde{z}^\lambda, \nabla \tilde{z}^\lambda_t, \tilde{v}^\lambda,  \tilde{v}^\lambda_t)||^2\nonumber\\
&  +K_\epsilon \lambda^4||(\tilde{E}^\lambda, \tilde{E}^\lambda_t, {\rm div}\tilde{E}^\lambda, {\rm div}\tilde{E}^\lambda_t,
\nabla{\rm div}\tilde{E}^\lambda, \nabla{\rm div}\tilde{E}^\lambda_t)||^2+K_\epsilon \lambda^4,
\end{align}
where we have used the facts ${\rm div} \tilde{v}^\lambda =0$ and ${\rm div} v=0$.
For the second integral, by   Cauchy-Schwartz's inequality,
Sobolev's embedding $H^2(\mathbb{T}^3)\hookrightarrow
L^\infty(\mathbb{T}^3)$, the inequality \eqref{eb2}, and  ${\rm div}
\tilde{v}^\lambda=0$, we have
\begin{align}\label{eb54}
&\quad \int\Big\{
\partial_t[\text{div}(\tilde{z}^\lambda\tilde{v}^\lambda)]
+\lambda^2\partial_t[\text{div}(\tilde{E}^\lambda\text{div}\tilde{E}^\lambda)]\Big\}\Delta\tilde{z}^\lambda_tdx\nonumber\\
& \leq \epsilon ||\Delta \tilde{z}^\lambda_t||^2
+K_\epsilon||\partial_t[\text{div}(\tilde{z}^\lambda\tilde{v}^\lambda)]||^2
+K_\epsilon \lambda^4||\partial_t[\text{div}(\tilde{E}^\lambda\text{div}\tilde{E}^\lambda)]||^2\nonumber\\
&\leq \epsilon ||\Delta \tilde{z}^\lambda_t||^2 +K_\epsilon(||\nabla
\tilde{z}^\lambda||^2_{H^1}||\tilde{v}^\lambda_t||^2_{H^1}
+||\nabla \tilde{z}^\lambda_t||^2 ||\tilde{v}^\lambda||^2_{L^\infty})\nonumber\\
&  \ \ +K_\epsilon \lambda^4(||
\tilde{E}^\lambda_t||^2_{L^\infty}||\nabla
\text{div}\tilde{E}^\lambda||^2+ 2||
\text{div}\tilde{E}^\lambda||^2_{H^1}||\text{div}\tilde{E}^\lambda_t||^2_{H^1}
+|| \tilde{E}^\lambda||^2_{L^\infty}||\nabla \text{div}\tilde{E}^\lambda_t||^2)\nonumber\\
&\leq \epsilon ||\Delta \tilde{z}^\lambda_t||^2 +K_\epsilon(||\nabla
\tilde{z}^\lambda||^2_{H^1}||\tilde{v}^\lambda_t||^2_{H^1}
+||\nabla \tilde{z}^\lambda_t||^2||\tilde{v}^\lambda||^2_{H^2})\nonumber\\
&  \ \ +K_\epsilon \lambda^4(||
\tilde{E}^\lambda_t||^2_{H^2}||\nabla
\text{div}\tilde{E}^\lambda||^2+ 2||
\text{div}\tilde{E}^\lambda||^2_{H^1}||\text{div}\tilde{E}^\lambda_t||^2_{H^1}
+|| \tilde{E}^\lambda||^2_{H^2}||\nabla \text{div}\tilde{E}^\lambda_t||^2)\nonumber\\
&\leq \epsilon ||\Delta \tilde{z}^\lambda_t||^2
+K_\epsilon|||\tilde{\mathbf{w}}^\lambda |||^4+ K_\epsilon\lambda^2|||\tilde{\mathbf{w}}^\lambda
|||^2\big(||\tilde{E}^\lambda_t||^2_{H^2}+||\text{div}\tilde{E}^\lambda_t||^2_{H^1}+
||\nabla \text{div}\tilde{E}^\lambda_t||^2\big).
  \end{align}
Thus, by putting  \eqref{eb52}-\eqref{eb54} together and taking
$\epsilon$ to be small enough,   we obtain
\begin{align}\label{eb55}
&   \frac{d}{dt}||\nabla \tilde{z}^\lambda_t||^2+c_{10}||\Delta\tilde{z}^\lambda_t||^2\nonumber\\
&\leq K  ||(\tilde{E}^\lambda_t, {\rm div}\tilde{E}^\lambda_t)||^2
+K_\epsilon||(\nabla\tilde{z}^\lambda, \nabla \tilde{z}^\lambda_t, \tilde{v}^\lambda,   \tilde{v}^\lambda_t)||^2\nonumber\\
& \quad +K  \lambda^4||(\tilde{E}^\lambda, \tilde{E}^\lambda_t, {\rm
div}\tilde{E}^\lambda, {\rm div}\tilde{E}^\lambda_t, \nabla{\rm
div}\tilde{E}^\lambda, \nabla{\rm div}\tilde{E}^\lambda_t)||^2\nonumber\\
& \quad +K |||\tilde{\mathbf{w}}^\lambda |||^4+ K \lambda^2|||\tilde{\mathbf{w}}^\lambda
|||^2\big(||\tilde{E}^\lambda_t||^2_{H^2}+||\text{div}\tilde{E}^\lambda_t||^2_{H^1}+
||\nabla \text{div}\tilde{E}^\lambda_t||^2\big)+K \lambda^4.
\end{align}

Differentiating \eqref{ea23} with respect to $t$, multiplying the resulting equation by ${\rm div} \tilde{E}^\lambda_t$ and
  integrating it over $\mathbb{T}^3$ with respect to $x$, we get
  \begin{align}\label{eb56}
&\frac{\lambda^2}{2}\frac{d}{dt}||{\rm div}\tilde{E}^\lambda_t||^2+\lambda^2||\nabla {\rm div}\tilde{E}^\lambda_t||^2
+\int Z |{\rm div}\tilde{E}^\lambda_t|^2 dx\nonumber\\
& \quad = -\int\big( \partial_t(\nabla Z\tilde{E}^\lambda)+Z_t{\rm
div} \tilde{E}^\lambda\big){\rm
div}\tilde{E}^\lambda_tdx-\lambda^2\int \partial_t
[\partial_t\text{div}\mathcal{E}-\Delta \text{div}\mathcal{E}]{\rm
div} \tilde{E}^\lambda_t dx\nonumber\\
&\quad \quad  -\int \partial_t
[\text{div}(\tilde{z}^\lambda\mathcal{E})]{\rm div}
\tilde{E}^\lambda_t dx
 -\lambda^2\int
\partial_t[\text{div}(\tilde{v}^\lambda\text{div}\mathcal{E}+v\text{div}\mathcal{E})]{\rm
div} \tilde{E}^\lambda_t dx\nonumber\\
&\quad\quad  +\int \partial_t [\text{div}(D\tilde{v}^\lambda)]{\rm
div} \tilde{E}^\lambda_t dx
 -\lambda^2\int \partial_t
[\text{div}(v\text{div}\tilde{E}^\lambda)]{\rm div}
\tilde{E}^\lambda_t dx
\nonumber\\
&\quad\quad-\int \partial_t [\text{div}(\tilde{z}^\lambda\tilde{E}^\lambda)]{\rm div} \tilde{E}^\lambda_t dx
 -\lambda^2\int \partial_t [\text{div}(\tilde{v}^\lambda{\rm div}\tilde{E}^\lambda)]{\rm div} \tilde{E}^\lambda_t dx
  \end{align}
We estimate each term on the right-hand side of \eqref{eb56}.
Noticing ${\rm div} \tilde{v}^\lambda =0$ and ${\rm div} v=0$, by   Cauchy-Schwartz's inequality,
using the regularity of $Z, \mathcal{E}, v $ and $D$,  the first six terms can be bounded by
\begin{align}\label{eb57}
& \epsilon||{\rm div}\tilde{E}^\lambda_t||^2 + K_\epsilon
||(\tilde{E}^\lambda, \tilde{E}^\lambda_t, {\rm
div}\tilde{E}^\lambda)||^2 +
K_\epsilon||(\tilde{z}^\lambda,\tilde{z}^\lambda_t, \nabla
\tilde{z}^\lambda, \nabla \tilde{z}^\lambda_t)||^2
\nonumber\\
&  +K_\epsilon ||(\tilde{v}^\lambda,\tilde{v}^\lambda_t)||^2
+K_\epsilon \lambda^4||(\tilde{v}^\lambda,\tilde{v}^\lambda_t)||^2
+K_\epsilon \lambda^4||(\nabla{\rm div}\tilde{E}^\lambda,\nabla{\rm div}\tilde{E}^\lambda_t)||^2+K_\epsilon
\lambda^4.
\end{align}
For the last two nonlinear terms,  by Cauchy-Schwartz's inequality,
  Sobolev's embedding $H^2(\mathbb{T}^3)\hookrightarrow
L^\infty(\mathbb{T}^3)$,  the inequality \eqref{eb2} and $\text{div}\tilde{v}^\lambda =0$,
 they can be estimated as follows
\begin{align}\label{eb58}
& -\int \partial_t [\text{div}(\tilde{z}^\lambda\tilde{E}^\lambda)]{\rm div} \tilde{E}^\lambda_t dx
 -\lambda^2\int \partial_t [\text{div}(\tilde{v}^\lambda{\rm div}\tilde{E}^\lambda)]{\rm div} \tilde{E}^\lambda_t dx\nonumber\\
& \leq \epsilon ||{\rm div} \tilde{E}^\lambda_t||^2
+K_\epsilon||\partial_t[\text{div}(\tilde{z}^\lambda\tilde{E}^\lambda)]||^2
+K_\epsilon \lambda^4||\partial_t[\text{div}(\tilde{v}^\lambda\text{div}\tilde{E}^\lambda)]||^2\nonumber\\
&\leq \epsilon ||{\rm div} \tilde{E}^\lambda_t||^2
 +K_\epsilon \big(|| \tilde{z}^\lambda||^2_{L^\infty}|| \text{div}\tilde{E}^\lambda_t||^2
+||  \tilde{z}^\lambda_t||^2_{L^\infty}|| \text{div}\tilde{E}^\lambda||^2+
|| \nabla \tilde{z}^\lambda||^2_{H^1}|| \tilde{E}^\lambda_t||^2_{H^1}\nonumber\\
&  \ \ +||\nabla \tilde{z}^\lambda_t||^2_{H^1}||\tilde{E}^\lambda||^2_{H^1}\big)
 +K_\epsilon\lambda^4(||\tilde{v}^\lambda_t||^2_{L^\infty}||\nabla {\rm div} \tilde{E}^\lambda||^2+
||\tilde{v}^\lambda||^2_{L^\infty}||\nabla {\rm div} \tilde{E}^\lambda_t||^2)\nonumber\\
&\leq \epsilon ||{\rm div} \tilde{E}^\lambda_t||^2
 +K_\epsilon \big(|| \tilde{z}^\lambda||^2_{H^2}|| \text{div}\tilde{E}^\lambda_t||^2
+||  \tilde{z}^\lambda_t||^2_{H^2}|| \text{div}\tilde{E}^\lambda||^2+
|| \nabla \tilde{z}^\lambda||^2_{H^1}|| \tilde{E}^\lambda_t||^2_{H^1}\nonumber\\
& \quad +||\nabla \tilde{z}^\lambda_t||^2_{H^1}||\tilde{E}^\lambda||^2_{H^1}\big)
 +K_\epsilon\lambda^4(||\tilde{v}^\lambda_t||^2_{H^2}||\nabla {\rm div} \tilde{E}^\lambda||^2+
||\tilde{v}^\lambda||^2_{H^2}||\nabla {\rm div} \tilde{E}^\lambda_t||^2)\nonumber\\
& \leq \epsilon ||{\rm div} \tilde{E}^\lambda_t||^2
+K_\epsilon|||\tilde{\mathbf{w}}^\lambda |||^2( || \text{div}\tilde{E}^\lambda_t||^2+
||  \tilde{z}^\lambda_t||^2_{H^2}+||  \tilde{E}^\lambda_t||^2_{H^1}+
||\nabla \tilde{z}^\lambda_t||^2_{H^1}\big)
\nonumber\\
&\quad + K_\epsilon\lambda^2|||\tilde{\mathbf{w}}^\lambda |||^2\big(
 ||\tilde{z}^\lambda_t||^2_{H^2}+
||\nabla \text{div}\tilde{E}^\lambda_t||^2\big)
\end{align}
Putting \eqref{eb56}-\eqref{eb58} together, using the positivity of $Z$,  and taking $
\epsilon$ small enough, we get
 \begin{align}\label{eb59}
&{\lambda^2}\frac{d}{dt}||{\rm div}\tilde{E}^\lambda_t||^2+2\lambda^2||\nabla {\rm div}\tilde{E}^\lambda_t||^2
+c_{11}\int   |{\rm div}\tilde{E}^\lambda_t|^2 dx\nonumber\\
&\leq   K
||(\tilde{E}^\lambda, \tilde{E}^\lambda_t, {\rm
div}\tilde{E}^\lambda)||^2 +
K_\epsilon||(\tilde{z}^\lambda,\tilde{z}^\lambda_t, \nabla
\tilde{z}^\lambda, \nabla \tilde{z}^\lambda_t)||^2
\nonumber\\
& \quad +K  ||(\tilde{v}^\lambda,\tilde{v}^\lambda_t)||^2
+K  \lambda^4||(\tilde{v}^\lambda,\tilde{v}^\lambda_t)||^2
+K  \lambda^4||(\nabla{\rm div}\tilde{E}^\lambda,\nabla{\rm div}\tilde{E}^\lambda_t)||^2\nonumber\\
&\quad  +K |||\tilde{\mathbf{w}}^\lambda |||^2( || \text{div}\tilde{E}^\lambda_t||^2+
||  \tilde{z}^\lambda_t||^2_{H^2}+||  \tilde{E}^\lambda_t||^2_{H^1}+
||\nabla \tilde{z}^\lambda_t||^2_{H^1}\big)
\nonumber\\
&\quad + K \lambda^2|||\tilde{\mathbf{w}}^\lambda |||^2\big(
 ||\tilde{z}^\lambda_t||^2_{H^2}+
||\nabla \text{div}\tilde{E}^\lambda_t||^2\big)+K
\lambda^4.
 \end{align}

 Differentiating \eqref{ea24} with respect to $t$, multiplying
the resulting equation by $-\Delta \tilde{v}^\lambda_t$,
  integrating it over $\mathbb{T}^3$ with respect to $x$ and using $\text{div} \tilde{v}^\lambda_t=0$, we get
\begin{align}\label{eb60}
& \ \ \frac12 \frac{d}{dt}||\nabla \tilde{v}^\lambda_t||^2+\mu ||\Delta\tilde{v}^\lambda_t||^2\nonumber\\
&\quad =  \int \partial_t( {v}\cdot \nabla \tilde{v}^\lambda)\Delta\tilde{v}^\lambda_t dx
  +\int \partial_t(\tilde{v}^\lambda\cdot \nabla  {v} )\Delta\tilde{v}^\lambda_t dx
  +\int \partial_t(D \tilde{E}^\lambda)\Delta\tilde{v}^\lambda_t dx\nonumber\\
  & \quad\  \ \ -\lambda^2 \int \partial_t (\mathcal{E}{\rm div}\mathcal{E})\Delta\tilde{v}^\lambda_t dx
-\lambda^2 \int \partial_t (\mathcal{E}{\rm div}\tilde{E}^\lambda)\Delta\tilde{v}^\lambda_t dx
  -\lambda^2 \int \partial_t (\tilde{E}{\rm div}\mathcal{E})\Delta\tilde{v}^\lambda_t dx\nonumber\\
  & \quad\  \ \ +
\int \partial_t(\tilde{v}^\lambda\cdot \nabla \tilde{v}^\lambda)\Delta\tilde{v}^\lambda_t dx
-\lambda^2 \int \partial_t (\tilde{E}^\lambda{\rm div}\tilde{E}^\lambda)\Delta\tilde{v}^\lambda_t dx.
   \end{align}
By the Cauchy-Schwartz's inequality and
using the regularity of $ v,D $ and $\mathcal{E}$,  the first six terms on the right-hand side of \eqref{eb60}
can be bounded by
\begin{align}\label{eb61}
& \epsilon ||\Delta\tilde{v}^\lambda_t||^2
+K_\epsilon||(\nabla\tilde{v}^\lambda,\nabla\tilde{v}^\lambda_t,\tilde{v}^\lambda,\tilde{v}^\lambda_t,
 \tilde{E}^\lambda, \tilde{E}^\lambda_t)||^2\nonumber\\
&\quad +K_\epsilon\lambda^4||(  \tilde{E}^\lambda,  \tilde{E}^\lambda_t)||^2
+K_\epsilon\lambda^4||({\rm div} \tilde{E}^\lambda,{\rm div} \tilde{E}^\lambda_t)||^2+K_\epsilon\lambda^4.
 \end{align}
By  Cauchy-Schwartz's inequality, Sobolev's embedding $H^2(\mathbb{T}^3)\hookrightarrow
L^\infty(\mathbb{T}^3)$, and
using the inequality \eqref{eb2},
 the last two nonlinear terms on the right-hand side of \eqref{eb60} can be treated as follows
\begin{align}\label{eb62}
&  \int \partial_t(\tilde{v}^\lambda\cdot \nabla \tilde{v}^\lambda)\Delta\tilde{v}^\lambda_t dx
-\lambda^2 \int \partial_t (\tilde{E}^\lambda{\rm div}\tilde{E}^\lambda)\Delta\tilde{v}^\lambda_t dx\nonumber\\
& \leq \epsilon ||\Delta \tilde{v}^\lambda_t||^2
+K_\epsilon||\partial_t(\tilde{v}^\lambda\cdot \nabla \tilde{v}^\lambda)||^2
+K_\epsilon \lambda^4||\partial_t(\tilde{E}^\lambda\text{div}\tilde{E}^\lambda)||^2\nonumber\\
&\leq \epsilon ||\Delta \tilde{v}^\lambda_t||^2
+K_\epsilon(|| \tilde{v}^\lambda_t||^2_{H^1}||\nabla\tilde{v}^\lambda||^2_{H^1}
+||  \tilde{v}^\lambda||^2_{L^\infty}||\nabla\tilde{v}^\lambda_t||^2)\nonumber\\
  &  \ \ +K_\epsilon \lambda^4(||
\tilde{E}^\lambda_t||^2_{H^1}||
\text{div}\tilde{E}^\lambda||^2_{H^1}+|| \tilde{E}^\lambda||^2_{L^\infty}||  \text{div}\tilde{E}^\lambda_t||^2)\nonumber\\
&\leq \epsilon ||\Delta \tilde{v}^\lambda_t||^2
+K_\epsilon(|| \tilde{v}^\lambda_t||^2_{H^1}||\nabla\tilde{v}^\lambda||^2_{H^1}
+||  \tilde{v}^\lambda||^2_{H^2}||\nabla\tilde{v}^\lambda_t||^2)\nonumber\\
  &  \ \ +K_\epsilon \lambda^4(||
\tilde{E}^\lambda_t||^2_{H^1}||
\text{div}\tilde{E}^\lambda||^2_{H^1}+|| \tilde{E}^\lambda||^2_{H^2}||  \text{div}\tilde{E}^\lambda_t||^2)\nonumber\\
&\leq \epsilon ||\Delta \tilde{v}^\lambda_t||^2
+K_\epsilon|||\tilde{\mathbf{w}}^\lambda |||^4
 .
  \end{align}
Thus, by putting  \eqref{eb60}-\eqref{eb62} together and taking
$\epsilon$ to be small enough,   we obtain
\begin{align}\label{eb63}
&   \frac{d}{dt}||\nabla \tilde{v}^\lambda_t||^2+\mu c_{12} ||\Delta\tilde{v}^\lambda_t||^2\nonumber\\
&\leq  K ||(\nabla\tilde{v}^\lambda,\nabla\tilde{v}^\lambda_t,\tilde{v}^\lambda,\tilde{v}^\lambda_t,
 \tilde{E}^\lambda, \tilde{E}^\lambda_t)||^2
+K \lambda^4||({\rm div} \tilde{E}^\lambda,{\rm div} \tilde{E}^\lambda_t)||^2\nonumber\\
&\quad +K \lambda^4||(  \tilde{E}^\lambda,  \tilde{E}^\lambda_t)||^2
+ K  |||\tilde{\mathbf{w}}^\lambda
|||^4 +K \lambda^4.
\end{align}

Combining \eqref{eb55}, \eqref{eb59} and \eqref{eb63},   and restricting $\lambda$ is small, we get
\begin{align}\label{eb64}
&\quad    \frac{d}{dt}\Big(\delta_7||\nabla \tilde{z}^\lambda_t||^2+\lambda^2||{\rm div}\tilde{E}^\lambda_t||^2
+\delta_8||\nabla \tilde{v}^\lambda_t||^2\Big)
 + \delta_7c_{10}||\Delta \tilde{z}^\lambda_t||^2+\delta_8\mu c_{12}||\Delta \tilde{v}^\lambda_t||^2\nonumber\\
&\quad     + (2\lambda^2 -K\lambda^4-\delta_7K\lambda^4)||\nabla {\rm div}\tilde{E}^\lambda_t||^2
+(c_{12}-\delta_7K-K\lambda^4(\delta_7+\delta_8))||{\rm div}\tilde{E}^\lambda_t||^2\nonumber\\
&  \leq K_4 ||(\tilde{z}^\lambda, \tilde{z}^\lambda_t, \nabla \tilde{z}^\lambda,\nabla \tilde{z}^\lambda_t)||^2
+K_4 ||(\tilde{v}^\lambda, \tilde{v}^\lambda_t, \nabla \tilde{v}^\lambda,\nabla \tilde{v}^\lambda_t)||^2\nonumber\\
&\quad     +K_4 ||(\tilde{E}^\lambda, \tilde{E}^\lambda_t, {\rm div} \tilde{E}^\lambda,\nabla {\rm div}\tilde{E}^\lambda)||^2
+K_4|||\tilde{\mathbf{w}}^\lambda|||^4\nonumber\\
&\quad
+K_4|||\tilde{\mathbf{w}}^\lambda|||^2
\big(
|| \tilde{z}^\lambda_t||^2_{H^2}+||\nabla\tilde{z}^\lambda_t||^2_{H^1}+ || \tilde{E}^\lambda_t||^2_{H^1}
+||{\rm div}\tilde{E}^\lambda_t||^2\big)\nonumber\\
&\quad    +K_4\lambda^2|||\tilde{\mathbf{w}}^\lambda|||^2
\big(
|| \tilde{E}^\lambda_t||^2_{H^2}+||{\rm div}\tilde{E}^\lambda_t||^2_{H^1}+
||\nabla{\rm div}\tilde{E}^\lambda_t||^2
 \big)+K_4\lambda^4,
\end{align}
for some $\delta_7$ and $ \delta_8$ sufficient small, which give the inequality \eqref{eb51}.
\end{proof}

\section{Proof of Theorem \ref{th}} \label{S3}
In this section, we  will use the energy estimates obtained in Section \ref{S2} to
establish the   entropy production integration inequality and compete the proof of our main result.
First, under the assumption of Theorem \ref{th}, by the standard elliptic regularity estimates, we have
\begin{align}
  ||\tilde{z}^\lambda||^2_{H^2}&  \leq K (||\tilde{z}^\lambda||^2+||\Delta \tilde{z}^\lambda||^2),\label{ec1}\\
  ||\tilde{z}^\lambda_t||^2_{H^2}&\leq K (||\tilde{z}^\lambda_t||^2+||\Delta \tilde{z}^\lambda_t||^2),\label{ec2}\\
  ||\tilde{v}^\lambda||^2_{H^2}&  \leq K (||\tilde{v}^\lambda||^2+||\Delta \tilde{v}^\lambda||^2),\label{ec3}\\
  ||\tilde{v}^\lambda_t||^2_{H^2}&\leq K (||\tilde{v}^\lambda_t||^2+||\Delta \tilde{v}^\lambda_t||^2),\label{ec4}\\
  ||\tilde{E}^\lambda||^2_{H^s}&  \leq K (||\tilde{E}^\lambda||^2+||{\rm div} \tilde{E}^\lambda||^2_{H^{s-1}}),s=1,2,\label{ec5}\\
  ||\tilde{E}^\lambda_t||^2_{H^s}&  \leq K (||\tilde{E}^\lambda_t||^2+||{\rm div} \tilde{E}^\lambda_t||^2_{H^{s-1}}),s=1,2.\label{ec6}
\end{align}

By the definitions of $\Gamma^\lambda (t)$ and
$|||\tilde{\mathbf{w}}^\lambda(t)|||$
 (see the definitions \eqref{es15}  and \eqref{eb1}
 above)
 and using the inequalities \eqref{ec1}-\eqref{ec6}, it is easy
to verify that there exist two constants $K_1$ and $K_2$, independent of $\lambda$, such that
\begin{equation}\label{ec7}
  K_1|||\tilde{\mathbf{w}}^\lambda(t)|||^2\leq \Gamma^\lambda (t)\leq K_2 |||\tilde{\mathbf{w}}^\lambda(t)|||^2.
\end{equation}

Using the inequalities
\eqref{eb3},\eqref{eb20},\eqref{eb35},\eqref{eb36}, and
\eqref{eb51}, we can obtain the new inequality  $[\eqref{eb3}+\delta
\eqref{eb20}+\delta^2
\eqref{eb35}]+\delta^3[\delta\eqref{eb36}+\eqref{eb51}]$. By taking
$\delta$ small enough, restricting $\lambda$ sufficient small, and a
tedious but straightforward computation, we obtain the following
relative entropy production integration inequality
\begin{align}\label{ec9}
 \quad \Gamma^\lambda(t)+ K\int^t_0G^\lambda (s)ds
 &\leq K\bar\Gamma^\lambda(t=0)+ K  (\Gamma^\lambda(t))^2+K\lambda^4+ K\int^t_0  \Gamma^\lambda(s)G^\lambda(s) ds\nonumber\\
& +K \int^t_0
\big\{\Gamma^\lambda(s)+(\Gamma^\lambda(s))^2\big\}(s)ds
  ,
\end{align}
where $G^\lambda(t)$ is defined  by \eqref{es16} and
\begin{align}\label{ec10}
  \bar\Gamma^\lambda(t=0)= &\big[||\tilde{z}^\lambda||^2+||\tilde{v}^\lambda||^2 +\lambda^2||\tilde{E}^\lambda||^2
  + ||\tilde{z}^\lambda_t||^2 +||\tilde{v}^\lambda_t||^2 +\lambda^2||\tilde{E}^\lambda_t||^2\big](t=0)\nonumber\\
  &+ (||\nabla \tilde{z}^\lambda_t||^2+\lambda^2||{\rm
div}\tilde{E}^\lambda_t||^2
+ ||\nabla \tilde{v}^\lambda_t||^2)(t=0).
\end{align}

The inequality \eqref{ec9} is a generalized Gronwall's type with an extra integration term, we have the following result.
 \begin{lem}\label{L7}
Suppose that
 \begin{equation}\label{ec11}
 \bar\Gamma^\lambda(t=0)\leq \bar K \lambda^2,
\end{equation}
where $\bar K$ is a positive constant, independent of $\lambda$.  Then
for any $T\in (0,T_{\max})$, $T_{\max}\leq +\infty$, there exists a positive constant $\lambda_0\ll 1$ such that for any
$\lambda\leq \lambda_0$ the inequality
\begin{align}\label{ec12}
\Gamma^\lambda(t)\leq \bar K \lambda^{2-\sigma}
\end{align}
holds for any $\sigma \in (0,2)$ and $0\leq t\leq T$.
 \end{lem}

 Since the  proof of  Lemma \ref{L7} is similar to that of Lemma 10
  in \cite{HW06}, we omit it here and continue our proof of
Theorem \ref{th}. In order to apply Lemma \ref{L7}, we need to verify \eqref{ec11}. In fact,
by the assumptions \eqref{ea26} on the initial data $(n^\lambda_0,
 p^\lambda_0, v^\lambda_0)$, we get $\tilde{E}^\lambda(t=0)=0$ since the solution involved here is smooth,
in particular, the solution  and its derivatives are continuous with respect to
 $x$ and $t$. Then, by using  the assumption \eqref{ea26},  $\tilde{E}^\lambda(t=0)=0$,  the continuity of the solution and its derivatives,
 and the equations \eqref{ea22}-\eqref{ea25}, we get
 \begin{align*}
  &\big[||\tilde{z}^\lambda||^2+||\tilde{v}^\lambda||^2
  + ||\tilde{z}^\lambda_t||^2 +||\tilde{v}^\lambda_t||^2 +\lambda^2||\tilde{E}^\lambda_t||^2\big](t=0)\nonumber\\
  &+ (||\nabla \tilde{z}^\lambda_t||^2+\lambda^2||{\rm
div}\tilde{E}^\lambda_t||^2
+ ||\nabla \tilde{v}^\lambda_t||^2)(t=0)\leq \bar K \lambda^2,
 \end{align*}
which gives the inequality \eqref{ec11}. Thus, by Lemma \ref{L7},
 the inequality  \eqref{ec12} holds. We easily get  the estimate \eqref{ea27} by the definition of
 $\Gamma^\lambda(t)$, the inequality \eqref{ec12},
  and the transform \eqref{ea21},
  which complete the proof of Theorem \ref{th}.

\vskip 5mm {\bf Acknowledgements}\ \ The author would like to
express his gratitude to  Dr. Jishan Fan for his   valuable
suggestions and careful reading of the first draft of this paper.
This work is  supported   by the National Natural Science Foundation
of China  (Grant   10501047).


\end{document}